\documentclass[journal]{IEEEtran}

\usepackage{graphicx}
\usepackage{subcaption}
\usepackage{amssymb,amsmath,amsthm}
\usepackage{mathtools}
\usepackage{amssymb,amsmath,amsfonts,amsthm,enumerate,color}
\newcommand{\vx}{{\mathbf{x}}}

\newcommand{\vA}{{\mathbf{A}}}

\newcommand{\vD}{{\mathbf{D}}}

\newcommand{\vH}{{\mathbf{H}}}
\newcommand{\vI}{{\mathbf{I}}}

\newcommand{\vP}{{\mathbf{P}}}

\newcommand{\vT}{{\mathbf{T}}}

\newcommand{\cN}{{\mathcal{N}}}

\newcommand{\cU}{{\mathcal{U}}}

\newcommand{\cX}{{\mathcal{X}}}

\newcommand{\EE}{\mathbb{E}}
\newcommand{\RR}{\mathbb{R}}

\newcommand{\vone}{{\mathbf{1}}}

\newcommand{\diag}{{\mathrm{diag}}} 
\newcommand{\grad}{{\nabla}}    

\newcommand{\prox}{{\mathbf{prox}}}
\DeclareMathOperator*{\argmin}{arg\,min} 
\newcommand{\St}{{\text{subject to}~}} 

\DeclareMathOperator*{\Min}{minimize}

\newcommand{\bdm}{\begin{displaymath}}
\newcommand{\edm}{\end{displaymath}}

\newcommand{\beq}{\begin{equation}}
\newcommand{\eeq}{\end{equation}}

\newcommand{\bfl}{\begin{flushleft}}
\newcommand{\efl}{\end{flushleft}}

\newcommand{\beqn}{\begin{eqnarray}}
\newcommand{\eeqn}{\end{eqnarray}}

\newcommand{\beqs}{\begin{align*}} 
\newcommand{\eeqs}{\end{align*}}  

\makeatletter
\newcommand{\Spvek}[2][r]{%
  \gdef\@VORNE{1}
  \left(\hskip-\arraycolsep%
    \begin{array}{#1}\vekSp@lten{#2}\end{array}%
  \hskip-\arraycolsep\right)}

\def\vekSp@lten#1{\xvekSp@lten#1;vekL@stLine;}
\def\vekL@stLine{vekL@stLine}
\def\xvekSp@lten#1;{\def\temp{#1}%
  \ifx\temp\vekL@stLine
  \else
    \ifnum\@VORNE=1\gdef\@VORNE{0}
    \else\@arraycr\fi%
    #1%
    \expandafter\xvekSp@lten
  \fi}
\makeatother



\usepackage[table]{xcolor}
\usepackage{colortbl}
\usepackage{multirow}
\usepackage{hhline}

\usepackage{epsfig} 
\usepackage{cite}
\usepackage{upgreek}
\usepackage{dsfont}
\usepackage{amssymb}
\usepackage{url}
\newtheorem{assumption}{Assumption}
\newtheorem{theorem}{Theorem}
\newtheorem{lemma}{Lemma}

\allowdisplaybreaks

\newcommand{\NN}{\mathbb{N}}
\newcommand{\tran}{^{\mathsf{T}} }

\newcommand{\sy}{{\scriptstyle{\mathcal{Y}}}}
\newcommand{\ssy}{{\scriptscriptstyle{\mathcal{Y}}}}
\newcommand{\sz}{{\scriptstyle{\mathcal{Z}}}}
\newcommand{\ssz}{{\scriptscriptstyle{\mathcal{Z}}}}
\newcommand{\sw}{{\scriptstyle{\mathcal{W}}}}

\usepackage{amsmath}

\begin{document}

\title{\LARGE \bf Walkman: A Communication-Efficient Random-Walk Algorithm \\ for Decentralized Optimization}

\author{Xianghui Mao$^\diamond$, Kun Yuan$^{*}$, Yubin Hu$^\diamond$, Yuantao Gu$^\diamond$, Ali H. Sayed$^\ddagger$, and Wotao Yin$^\dagger$
\thanks{$^{\diamond}$X. Mao, Y. Hu and Y. Gu are with the Electrical Engineering Department, Tsinghua University, Beijing, China. Email: maoxh14@mails.tsinghua.edu.cn, hu-yb16@mails.tsinghua.edu.cn, gyt@tsinghua.edu.cn. Their work was partly funded by the National Natural Science Foundation of China (NSFC 61571263, 61531166005) and the National Key Research and Development Program of China (Project No. 2016YFE0201900, 2017YFC0403600).
}%
\thanks{$^{*}$K. Yuan is with the Electrical and Computer Engineering Department, University of California, Los Angeles, CA 90095. Email:{\small kunyuan}@ucla.edu. 
}%
\thanks{$^{\ddagger}$A. H. Sayed is with the Ecole Polytechnique Federale de Lausanne
(EPFL), School of Engineering, CH-1015 Lausanne, Switzerland. Email:
        {\small ali.sayed@epfl.ch}. His work was supported in part by NSF grant CCF-1524250.}
\thanks{$^{\dagger}$W. Yin is with the Mathematics Department, University of California, Los Angeles, CA 90095. Email: {\small wotaoyin@math.ucla.edu. His work was supported in part by NSF DMS-1720237 and ONR N000141712162.}}    
}

\maketitle

\begin{abstract}
This paper addresses consensus optimization problems in a multi-agent network, where all agents  collaboratively find a minimizer for the sum of their private functions. We develop a new decentralized algorithm in which each agent communicates only with its neighbors.

State-of-the-art decentralized algorithms  use communications between either {\em all pairs of adjacent agents} or a {\em random subset} of them at each iteration. Another class of algorithms uses a \emph{random walk incremental} strategy, which sequentially activates a succession of nodes; these  incremental algorithms require diminishing step sizes to converge to the  solution, so their convergence is relatively slow.

In this work, we propose a random walk algorithm that uses a fixed step size and converges faster than the existing random walk incremental algorithms. Our algorithm is also communication efficient. Each iteration uses only one link to communicate the latest information for an agent to another.
Since this communication rule mimics a man walking around the network, we call our new algorithm \emph{Walkman}. 
 We establish convergence for convex and nonconvex objectives. For decentralized least squares, we derive a linear rate of convergence and obtain a better communication complexity than those of other decentralized algorithms. Numerical experiments verify our analysis results.
\end{abstract}


\IEEEpeerreviewmaketitle


\section{Introduction}
\label{sec-introduction}
Consider a directed graph $G=(V,E)$, where $V=\{1,2,\ldots,n\}$ is the set of agents and $E$ is the set of $m$ edges.
We aim to solve the following optimization problem:
\begin{align}\label{eq:mainprob}
\Min_{x\in\RR^p}\quad r(x) + \frac{1}{n}\sum_{i=1}^n f_i(x),
\end{align}
where each $f_i$ is locally held by agent $i$ and $r$ is a globally known regularizer. Both $f_i$ and $r$ can be non-convex.  
An algorithm is decentralized if it relies only on communications between neighbors (adjacent agents); there is no central node that collects or distributes  information to the agents. 
Decentralize consensus optimization finds applications in various areas including wireless sensor networks, multi-vehicle and multi-robot control systems, smart grid implementations, distributed adaptation and estimation \cite{sayed2014adaptive,sayed2014adaptation}, distributed statistical learning \cite{duchi2012dual,chen2015dictionary,chouvardas2012sparsity} and clustering \cite{zhao2015distributed}.

\subsection{The literature}
There are several decentralized numerical approaches to solve problem \eqref{eq:mainprob} or its special case without the regularizer $r$. One well-known approach lets {\em every} agent exchange information with all, or a random subset, of its direct neighbors per iteration. This is illustrated in Fig. \ref{fig:gossip_comm}, where agent $i$ is collecting information from all its neighbors (to update its local variables). This approach includes well-known algorithms such as diffusion~\cite{sayed2014adaptation,sayed2014adaptive} and consensus~\cite{nedic2009distributed,yuan2016convergence}, distributed ADMM (D-ADMM)~\cite{mateos2010distributed,mota2013d,shi2014linear,chang2014multi,aybat2017distributed}, EXTRA~\cite{shi2015extra}, PG-EXTRA~\cite{shi2015proximal}, DIGing~\cite{nedic2017achieving}, exact diffusion~\cite{yuan2017exact1}, NIDS~\cite{li2017decentralized}, and beyond. 
{\color{black} Among them, Push-Sum \cite{tsianos2012push},  EXTRAPUSH \cite{zeng2015extrapush} and subgradient-push \cite{nedic2016stochastic} are designed for directed graphs, while DIGing is for time-varying graphs.}
These algorithms have good convergence rates in the number of iterations. D-ADMM, EXTRA, DIGing, exact diffusion, and NIDS all converge linearly to the exact solution assuming strong convexity and using constant step-sizes. Their communication per iteration is relatively high.
Depending on the density of the network,  the costs are $O(n)$ computation and $O(n)$--$O(n^2)$ communications per iteration. 

\begin{figure}
\begin{minipage}[t]{0.48\linewidth}
\centering
\includegraphics[width=1.5in]{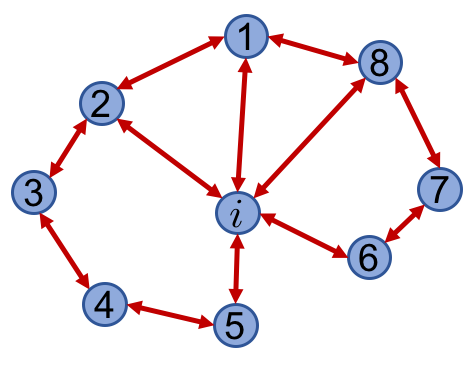}
\caption{\footnotesize communications in the $k$-th iteration for gossip type methods.}
\label{fig:gossip_comm}
\end{minipage}%
\begin{minipage}[t]{0.48\linewidth}
\centering
\includegraphics[width=1.5in]{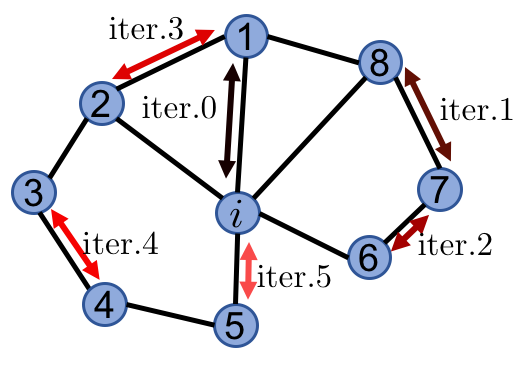}
\caption{\footnotesize  communications in $5$ adjacent iterations for randomized gossip type methods.}
\label{fig:randomized_gossip_comm}
\end{minipage}
\begin{minipage}[t]{0.48\linewidth}
\centering
\includegraphics[width=1.5in]{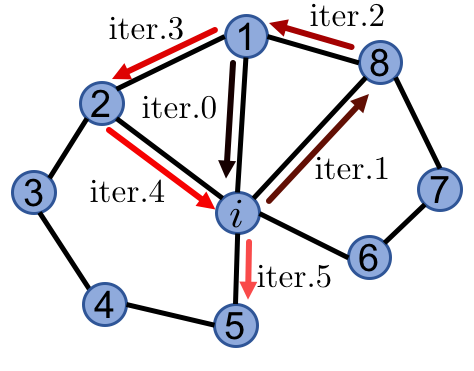}
\caption{\footnotesize communications in $5$ adjacent iterations for random walk based methods.}
\label{fig:rw_comm}
\end{minipage}
\begin{minipage}[t]{0.48\linewidth}
\centering
\includegraphics[width=1.5in]{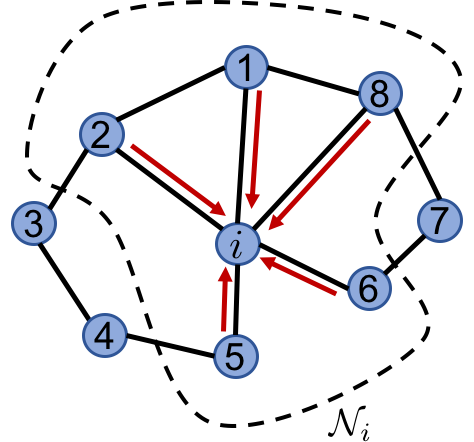}
\caption{\footnotesize communications in the $k$-th iteration for RW-ADMM.}
\label{fig:rw_admm_comm}
\end{minipage}
\end{figure}


{\color{black} To alleviate the communication burden of decentralized optimization methods, another line of works \cite{boyd2006randomized, cao2006accelerated, hendrikx2018accelerated} study the communication pattern illustrated in Fig. \ref{fig:randomized_gossip_comm}, specifically, randomly activating one edge for bi-directional communication in each iteration. Among them, randomized gossip algorithms proposed  in \cite{boyd2006randomized, cao2006accelerated} are designed to solve average consensus problem. More recently, ESDACD \cite{hendrikx2018accelerated} implements such random activation to solve general smooth strongly-convex consensus problem. In general, the selected edges are not continuous, some global coordination is required to ensure non-overlapping of iterations.}

Another approach is based on the (random) walk (sub)gradient method~\cite{bertsekas1997new,ram2009incremental,johansson2009randomized, lopes2007incremental,lopes2010randomized}, where a variable $x$ will move through a (random) succession of agents in the network. At each iteration, the agent $i$ that receives $x$ updates $x$ using one of the subgradients of $f_i$, followed by sending $x$ to a (random) neighbor. Fig. \ref{fig:rw_comm} illustrates the communications along a walk $(1,i,8,1,2,i,5,\cdots)$. Since only one node and one link are used at each iteration, this approach only costs $O(1)$ computation and $O(1)$ communication per iteration. 
{\color{black} Thanks to the natural continuity of random walk, it is easy for the involved agents to coordinate.}
 The works  \cite{lopes2007incremental,lopes2010randomized} apply random walks in the context of adaptive networks and relies on stochastic gradients. 
If these algorithms {\color{black} use} a constant step-size,  their iterates converge to a neighborhood of the solution. 
If the step-size is small, the neighborhood will be proportionally small but convergence becomes slow. For applications where convergence to the exact solution is required, decaying step-sizes must be used, which leads to slow convergence.
{\color{black} The authors of recent work \cite{shah2018linearly} study a
mixture of each node exchanging information with all pattern
and random walk pattern  as shown in Fig. \ref{fig:rw_admm_comm}, and propose RW-ADMM, where each node in the random walk starts computing after collecting information from all its neighbors. RW-ADMM is proved to converge under constant stepsize on the sacrifice of more communication per iteration.}



{\color{black}\subsection{Contribution}}
\label{sec-intro-contribution}
In this paper, we propose a new random walk algorithm for decentralized consensus optimization that uses a \emph{fixed step-size} and converges to the \emph{exact solution}. It is significantly faster than the existing random-walk (sub)gradient incremental methods. 

{\color{black} When both $r$ and $f_i$ are possibly non-convex and $f_i$ are Lipschitz differentiable, we show that the iterates $x^k$ generated by Walkman will converge to the stationary point $x^\star$ almost surely. In addition, we establish a linear convergence rate for decentralized least squares.

Walkman is communication efficient. For decentralized least squares, the communication complexity of Walkman compares favorably with existing popular algorithms. The result is listed in Table \ref{tb:comm_compar}. {\color{black} Consider a network with transition probability matrix $\vP\in \RR^{n\times n}$ where $[\vP]_{ij} = p(i_{k+1} = j | i_k = i) \in [0,1]$. } We show that, if
\begin{align}\label{cond}
\lambda_2(\vP) \le 1 - \frac{\ln^{4/3}(n)}{m^{2/3}} \approx 1 - \frac{1}{m^{2/3}},
\end{align}
which implies the connectivity of the network is moderate or better, then our algorithm uses less communication  than {\em all} the state-of-the-art decentralized algorithms listed in the table. 

Our  simulation results support the claimed communication efficiency of  Walkman in least squares and other problems. 
}

\begin{table}[h]
        \vspace{1mm}
        \centering
        \begin{tabular}{c|c}
                {\bf Algorithm} & {\bf Communication Complexity} \\\hline
                \hline Walkman (proposed) & $O\bigg( {\ln\left(\frac{1}{\epsilon}\right)} \cdot  {\frac{n\ln^3(n)}{(1-\lambda_2(\vP))^2}}\bigg)$ \\
                \hline D-ADMM\cite{shi2014linear} & $O\left( {\ln\left(\frac{1}{\epsilon}\right)} \cdot  {\big(\frac{m}{(1-\lambda_2(\vP))^{1/2}}\big)}\right)$\\
                \hline EXTRA\cite{shi2015extra} & $O\left( {\ln\left(\frac{1}{\epsilon}\right)} \cdot  {\big(\frac{m}{1-\lambda_2(\vP)}\big)}\right)$\\
                \hline Exact diffusion\cite{yuan2017exact1} & $O\left( {\ln\left(\frac{1}{\epsilon}\right)} \cdot  {\big(\frac{m}{1-\lambda_2(\vP)}\big)}\right)$\\
                \hline
                {\color{black}  ESDACD\cite{hendrikx2018accelerated}} & $O\Big(\ln(\frac{1}{\epsilon})\cdot \frac{\sqrt{mn}}{\sqrt{(1-\lambda_2(\vP))}}\Big)$\\
                \hline 
                {\color{black}RW-ADMM \cite{shah2018linearly}} & $O\Big(\ln\left(\frac{1}{\epsilon}\right)\cdot \frac{m^2}{n\sqrt{(1-\lambda_2(\vP))}}\Big)$\\
                \hline
        \end{tabular}
        \caption{\small  Communication complexities of various algorithms {\color{black}when solving decentralized least squares problem} with $\lambda_2(\vP)$ is close to $1$. The network has $n$ nodes and $m$ arcs, $m\in [n, n(n-1)]$, with each node is connected by $m/n$ arcs. {\color{black}The quantity $\epsilon$ is the target accuracy, $\vP$ is the probability transition matrix, and $\lambda_2(\vP)$ is the second largest eigenvalue of $\vP$, one of the measures of the connectivity of the network.}}\label{tb:comm_compar}
        \vspace{1mm}
\end{table}


\subsection{Discussion}
{\color{black}Walkman is a random-walk algorithm. Its efficiency depends on how long it takes the walk to visit all the agents. This is known as the \emph{cover time}. When Walkman only needs visit every agent at least once (which is the case to compute the consensus average),  the cover time is exactly the  complexity of Walkman. For  the cover times of random walks in various graphs, we refer the reader to  \cite[Chapter 11]{levin2017markov}.}

For more general problems, Walkman must visit each agent infinitely many times to converge. Its efficiency depends on how frequently all of the agents are revisited. For a random walk, this can be described by the \emph{mixing time} of the underlying Markov chain. Next, we present relevant assumptions.
\begin{assumption}\label{ass-markov}
The random walk $(i_k)_{k\ge 0}$, $i_k\in V$, forms an irreducible and aperiodic Markov chain with transition probability matrix $\vP\in \RR^{n\times n}$ where $[\vP]_{ij} = p(i_{k+1} = j | i_k = i) \in [0,1]$ and stationary distribution $\pi$ satisfying $\pi^T \vP = \pi^T$.
\end{assumption}
If the underlying network is a complete graph, we can choose $P$ so that $P_{ij}=p(i_{k+1} = j | i_k = i)=\frac{1}{n}$ for all $i,j\in V$, a case analyzed in \cite[\S2.6.1]{Peng_2015_AROCK} (barring asynchronicity therein). For a more general network that is connected, we need the \emph{mixing time} (for given $\delta>0$), which is defined as the smallest integer $\uptau(\delta)$ such that, for all $i\in V$,
\begin{align}
\big\|[\vP^{\uptau(\delta)}]_{i,:} - \pi\tran\big\|\le\delta\pi_*,\label{eq:ori_mix}
\end{align}
where $\pi_* := \Min_{i\in V} \pi_i$, and $[\vP^{\uptau(\delta)}]_{i,:}$ denotes the $i$th row of $\vP^{\uptau(\delta)}$.
This inequality states: regardless of current state $i$ and time $k$,  the probability of visiting each state $j$ after $\uptau(\delta)$ more steps is $(\delta \pi_*)$-close to $\pi_j$, that is, for all $i, j\in V$,
\begin{align}
\big|[\vP^{\uptau(\delta)}]_{ij} - \pi_j\big| \le \delta\pi_*. \label{eq:mix_time_prop}
\end{align}

A good reference for mixing time is \cite{levin2017markov}.
The mixing time requirement, inequality \eqref{eq:ori_mix}, is guaranteed to hold for\footnote{Here is a trivial proof. For any $k\ge 1$, by definition, it holds $[\vP^{k}]_{i,:} -\pi\tran = \left([\vP^{k-1}]_{i,:} -\pi\tran\right)\vP$, and $\left([\vP^{k}]_{i,:} -\pi\tran\right)\vone=0$. Hence,  $\|[\vP^{k}]_{i,:} - \pi\tran\|\le \|[\vP^{k-1}]_{i,:} - \pi\tran\|\sigma(\vP)\le \cdots\le \|\vI_{i,:} -\pi\tran\|\sigma^k(\vP).$
We can bound $\|\vI_{i,:} - \pi\tran\|^2\le (1-\pi_i)^2 +\sum_{j\neq i}\pi_j^2 \le (1-\pi_*)^2 + (1-\pi_*)^2 =2 (1-\pi_*)^2$.
 Therefore, by ensuring $\sqrt{2}(\sigma(\vP))^{\uptau(\delta)}(1-\pi_*)\le \delta\pi_*$, which simplifies to condition \eqref{eq:def_J} by Taylor series, we  guarantee \eqref{eq:ori_mix} to hold.}
\begin{align}
\uptau(\delta) := \Big\lceil\frac{1}{1-\sigma(\vP)}\ln\frac{\sqrt{2}}{\delta \pi_*}\Big\rceil\label{eq:def_J}
\end{align}
for  $\sigma(\vP) := \sup\big\{\| f\tran \vP\|/\|f\|:f\tran \vone = 0,f\in\RR^n\big\}.$

We will use inequality \eqref{eq:mix_time_prop}  to show the sufficient descent of a Lyapunov function $L^k$, which was used  in \cite{hong2016convergence} and extended in \cite{WangYinZeng2015_global}. However, the analyses in \cite{hong2016convergence,WangYinZeng2015_global} only help us show $L^k\ge L^{k+1}$ and the existence of a lower bound. Because a random walk $(i_k)_{k\ge 0}$ is neither essentially cyclic nor i.i.d. random (except for complete graphs), we must use a new analytic technique, which is motivated by the recent paper \cite{sun2018Markov}. This new technique integrates mixing-time bounds with a conventional line of convergence analysis.

For decentralized least squares, we give the communication complexity bound of Walkman  in term of $\sigma(\vP)$. This quantity also determines the  communication complexity bounds of D-ADMM, EXTRA, and exact diffusion. Therefore, we can compare their communication complexities. For moderately well connected networks, we show in \S\ref{sc:comm} that the bound of Walkman is the lowest.

Even though D-ADMM, EXTRA, and exact diffusion use more total communications, their communications over different edges in each iteration are concurrent, so they \emph{may} take less {\em total communication time}. However, this time will increase and even overpass the Walkman time if different edges have different communication latencies and bandwidths, and if synchronization overhead is included. In an ideal situation where every communication takes the same amount of time and synchronization has no overhead, Walkman is found to be slower in time, unsurprisingly.

Although this paper does not discuss data privacy, Walkman  protects privacy better than diffusion, consensus, D-ADMM, etc., since the communication path is random and only the current iterate $x^k$ is sent out by the active agent. It is difficult for an agent to monitor the computation of its neighbors.

The limitation of this paper lies in that the linear convergence rate analysis applies only to least squares (though convergence {\color{black} and  a sublinear convergence rate} are established for more general problems) and  that the transition matrix is stationary. They need more space to address in our future work. Another direction to generalize this work is to create multiple simultaneous random walks, which may reduce the total solution time. The information exchange across random walks will require careful design and analysis.

In the rest of this paper, \S\ref{sc:derive} derives  Walkman, \S\ref{sc:converg} presents the main convergence result and the key lemmas,  \S\ref{sc:ls} focuses on least squares and obtains its linear convergence rate of Walkman, \S\ref{sc:comm} analyzes communication complexities and make comparisons between Walkman and other algorithms, \S\ref{sc:num} presents numerical simulation results, and finally \S\ref{sc:con} summarizes the findings of this paper.

\section{Derivation of Walkman}\label{sc:derive}
{\color{black}Walkman  can be derived by  modifying existing algorithms to use a random walk, for example, ADMM~\cite{GlowinskiMarroco1975_LapproximationPar,GabayMercier1976_dual} or PPG \cite{RyuYin2017_proximalproximalgradient}.} By defining
\begin{align}
\sy:=\mathrm{col}\{y_1, y_2, \cdots, y_n\}\in \RR^{np},
 \ F(\sy) := \sum_{i=1}^{n}f_i(y_i),
\end{align}
we can compactly rewrite problem \eqref{eq:mainprob} as
\begin{align}
\Min_{x,\ {\footnotesize \ssy}}\quad &r(x) + \frac{1}{n}F(\sy), \nonumber \\
\St \quad& \mathds{1}\otimes x - \sy = 0, \label{eq:mainprob-equiv}
\end{align}
where $\mathds{1}=[1~1\dots 1]\tran \in \RR^n$ and $\otimes$ is the Kronecker product. The constraint is equivalent to $x-y_i=0$ for $i=1,\dots,n$. The augmented Lagrangian for problem \eqref{eq:mainprob-equiv} is
\begin{align}
L_{\beta}\hspace{-1mm}\left( x, \sy; \sz \right) := &r(x) + \frac{1}{n} \Big(  F(\sy) + \langle \sz, \mathds{1}\otimes x - \sy \rangle\notag\\
& + \frac{\beta}{2}\|\mathds{1}\otimes x- \sy\|^2 \Big), \label{eq-Lagrangian}
\end{align}
where $\sz:={\rm col}\{z_1,\cdots,z_n\} \in \RR^{np}$ is the dual variable (Lagrange multipliers) and $\beta >0$ is a constant parameter. The standard ADMM algorithm is an iteration that minimizes $L_{\beta}\left( x, \sy; \sz \right)$ in $x$, then in $\sy$, and finally updates $\sz$.  {\color{black}Applying ADMM to problem \eqref{eq:mainprob-equiv} yields (not our algorithm)
\begin{subequations} \label{eq:admm3'-full}
        \begin{align}
                        \bar{x}^{k+1} & = \frac{1}{n}\sum_{i=1}^n(y_i^{k}-\frac{z_i^k}{\beta}), \label{eq:admm3xa'-full}\\
                 x^{k+1} & = \prox_{\frac{1}{\beta}r}(\bar{x}^{k+1}),\label{eq:admm3xb'-full}\\
                 y_i^{k+1} & = \prox_{\frac{1}{\beta}f_i}\Big(x^{k+1} + \frac{z_i^k}{\beta}\Big), \quad \forall i\in V \label{eq:admm3y'-full}\\
                 z_i^{k+1} & = z_i^k + \beta(x^{k+1} - y_i^{k+1}),\quad \hspace{1.8mm} \forall i\in V  \label{eq:admm3z'-full}
        \end{align}
\end{subequations}
}\hspace{-2mm} where
the proximal operator 
is defined as
$\prox_{\gamma f}(x) := \argmin_{y} f(y) + \frac{1}{2\gamma}\|y - x\|_2^2.$
Since computing the sum in \eqref{eq:admm3xa'-full} needs information from all the agents, it is too expensive to realize in a decentralized fashion. However, if each ADMM iteration updates only $y_{i_k}$ and $z_{i_k}$ in \eqref{eq:admm3y'-full} and \eqref{eq:admm3z'-full}, keeping the remaining $\{y_{i}\}_{i\neq i_k}$, $\{z_{i}\}_{i\neq i_k}$ unchanged,
{\color{black} the algorithm then changes to: }
\begin{subequations} \label{eq:admm3'}
        \begin{align}
        x^{k+1} & = \prox_{\frac{1}{\beta}r}(\bar{x}^{k+1}), \label{eq:admm3x'}\noeqref{eq:admm3x'}\\
        y_i^{k+1} & =
        \begin{cases}
        \prox_{\frac{1}{\beta}f_i}(x^{k+1}+\frac{z_i^k}{\beta}),& i=i_k\\
        y_i^{k},& \text{otherwise}
        \end{cases}\label{eq:admm3y'}\noeqref{eq:admm3y'}\\
        z_i^{k+1} & =
        \begin{cases}
        z_i^k + \beta(x^{k+1} - y_i^{k+1}) ,& i=i_k\\
        z_i^k,& \text{otherwise}
        \end{cases}\label{eq:admm3z'}\\
        \label{eq-x-bar-update}
\bar{x}^{k+2}& = \bar{x}^{k+1} + \frac{1}{n}\big(y_{i_k}^{k+1} - \frac{z_{i_k}^{k+1}}{\beta}\big) - \frac{1}{n}\big(y_{i_k}^{k} - \frac{z_{i_k}^{k}}{\beta}\big).
        \end{align}
\end{subequations}
If we initialize $\{y_i^0\}_{i=1}^n$ and $\{z_i^0\}_{i=1}^n$ so that
\begin{align}\label{initialization}
        \bar{x}^{1} = \frac{1}{n}\sum_{i=1}^n(y_i^{0}-\frac{z_i^0}{\beta}) = 0,
\end{align}
for example, by simply setting $y_i^0 = 0$ and $z_i^0 = 0$, $i=1,\cdots,n$, then {\color{black} with only the $i_k$-th part of variables $\sy$ and $\sz$ updated in each \eqref{eq:admm3'}, mathematical induction implies that } \eqref{eq-x-bar-update}  automatically maintains
\begin{align*}\bar{x}^{k+1} & = \frac{1}{n}\sum_{i=1}^n(y_i^{k}-\frac{z_i^k}{\beta}).\end{align*} 
{\color{black} Note that,  the second equation of the initialization condition in \eqref{initialization} can be conducted locally, whereas the constraint on $\bar{x}^1$ only involves the agent where the random walk starts. Therefore, a simple initialization satisfying \eqref{initialization}  can be realized without any ``consensus''-type preprocessing. }
We call \eqref{eq:admm3'}  Walkman. Its decentralized implementation  is presented in Algorithm 1. The variable {$\bar{x}^k$} is updated by agent $i_k$ and passed as a token to agent $i_{k+1}$. 

\begin{table}[h!]
        \noindent \hrule height 1pt
        \vspace{1mm}
        \noindent \textbf{\normalsize Algorithm 1: Walkman}
        \vspace{1mm}
        \hrule height 1pt
        \vspace{1mm}
        \noindent
        \textbf{\normalsize Initialization:} \normalsize initialize $y_i^0$ and $z_i^0$ so that \eqref{initialization} holds;\\
        \textbf{\normalsize { Repeat}} \normalsize for $k=0,1,2,\dots$ until convergence

        \normalsize {\color{white}an} agent $i_k$ do:\\
        \normalsize{\color{white}world} update $x^{k+1}$ according to \eqref{eq:admm3x'};\\
        \normalsize{\color{white}world} update $y_{i_k}^{k+1}$ according to \eqref{eq:admm3y'} or \eqref{eq:admmy-linearized};\\
        \normalsize{\color{white}world} update $z_{i_k}^{k+1}$ according to \eqref{eq:admm3z'};\\
        \normalsize{\color{white}world} update $\bar{x}^{k+2}$ according to \eqref{eq-x-bar-update}; \\
        {\color{white}world} send $\bar{x}^{k+2}$ via edge $(i_{k},i_{k+1})$ to agent $i_{k+1}$;\\
        \textbf{{End}}\\[-2mm]
        \hrule
        \vspace{-1mm}
\end{table}

\noindent \textbf{Use $\nabla f_i$ instead of $\prox_{f_i}$.}
If the regularizer $r$ is proximable, i.e., $\prox_{\gamma r}$ can be computed in $O(n)$ or $O(n{\rm polylog}(n))$ time, the computational resources are mainly consumed on solving the minimization problem in step \eqref{eq:admm3y'}.
We can avoid it  by using the cheaper gradient descent, like in diffusion, consensus, EXTRA, DIGing, exact diffusion, and NIDS. If $f_i$ is differentiable, 
we replace \eqref{eq:admm3y'} with the update: 
        \begin{align}\label{eq:admmy-linearized}\tag{\ref{eq:admm3y'}'}
                y_{i}^{k+1} =
                \begin{cases}
                        x^{k+1} + \frac{1}{\beta}z_i^k - \frac{1}{\beta} \grad f_i (y_i^k), & i = i_k\\
                        y_i^k, & \mbox{otherwise.}
                \end{cases}
        \end{align}
Compare to \eqref{eq:admm3y'}, update \eqref{eq:admmy-linearized} saves computations but can  cause more iterations and thus more total communications. One can choose between {\eqref{eq:admm3y'}} and \eqref{eq:admmy-linearized} based on computation and communication tradeoffs in applications.  {\color{black} In the next section, we are going to analyze their performance.}

\section{Convergence}\label{sc:converg}
In this section we present convergence of Walkman
based on  the following assumptions. 
\begin{assumption}\label{coercivity}
The objective function in original problem \eqref{eq:mainprob}, $r(x) + \frac{1}{n}\sum_{i=1}^n  f_i(x)$, is bounded from below over $\RR^p$ {\color{black} ( let $\underline{f}$ denote the lower bound)} and is coercive over $\RR^p$, that is, $r(x) + \frac{1}{n}\sum_{i=1}^n  f_i(x)\to\infty$ for any sequence $x^k\in\RR^p$ and $\|x^k\|\to\infty.$
\end{assumption}
Assumption \ref{coercivity} is \emph{not} over $\RR^{np}$ but $\RR^p$, so it  is easy to satisfy.
\begin{assumption}\label{ass-lipschitz}
        Each $f_i(x)$ is $L$-Lipschitz differentiable, that is, for any $u,v\in\RR^p$,
        \begin{align}
        \|\nabla f_i(u) - \nabla f_i(v)\| \le L\|u-v\|, \quad i = 1, \dots, n. \label{eq-ass-lip}
        \end{align}
\end{assumption}
\begin{assumption}\label{ass-prox}
        The lower semi-continuous function $r(x)$ is $\gamma$-semiconvex, that is, $r(\cdot)+\frac{\gamma}{2}\|\cdot\|^2$ is convex or equivalently,
        \begin{align}
                r(y)+\frac{\gamma}{2}\|y-x\|^2\geq r(x)+\langle d, y-x\rangle, \forall x, y,\forall d\in \partial r(x). \label{eq-prox-regular}
        \end{align}
\end{assumption}
We first introduce the notation used in our analysis.
The first time that the Markov chain $(i_k)_{k\ge 0}$ hits agent $i$ is denoted as $T_i:=\min\{k: i_k = i\}$, and their max over $i$ is
\begin{equation}\label{eq:T_def}
T:=\max\lbrace T_1, \cdots, T_n\rbrace.
\end{equation} 
By iteration $T$, every agent has been visited at least once. Based on Assumption \ref{ass-markov}, the Markov chain is positive recurrent and, therefore, $\Pr(T< \infty) = 1.$
For  $k>T$, let $\tau (k,i)$ denote the  iteration of the last visit to agent $i$ before $k$, that is, 
\begin{equation}\label{eq:tau_def}
\tau (k,i):=\max\lbrace k':i_{k'}=i,k'<k \rbrace.
\end{equation} 

Next, we define two separate Lyapunov functions for Walkman updating $\sy$ using \eqref{eq:admm3y'} (computing $\prox_{\frac{1}{\beta}f_i}$) and  \eqref{eq:admmy-linearized} (computing $\grad f_i (y_i^k)$):
\begin{align}
        L_{\beta}^k &:= L_\beta(x^k, \sy^k; \sz^k),\label{eq:defLk} \\
        M_{\beta}^k &:= L^k_\beta + \frac{L^2}{n}\sum_{i=1}^{n}\| y_{i}^{\tau (k,i)+1} - y_{i}^{\tau (k,i)}\|^2,\label{eq:defMk}
\end{align}
where $L_{\beta}\hspace{-1mm}\left( x, \sy; \sz \right)$ is defined  in \eqref{eq-Lagrangian}.
We establish the descent of $L_\beta^k$ (resp. $M_{\beta}^k$) for Walkman using \eqref{eq:admm3y'} (resp.  \eqref{eq:admmy-linearized}).
\begin{lemma}\label{lm:1}
        Under Assumptions \ref{coercivity}, \ref{ass-lipschitz}, and \ref{ass-prox},  the iterates $(x^k, \sy^k, \sz^k)_{k\ge 0}$ generated by Walkman \eqref{eq:admm3'}, or Algorithm 1, satisfy the following  properties:
        \begin{enumerate}
                \item  for \eqref{eq:admm3y'} and $\beta\ge\max\{\gamma, 2L+2\}$,  $(L_\beta^k)_{k\ge 0}$ is lower bounded and convergent;
                \end{enumerate}
                \begin{enumerate}[~~1')]
                \item for \eqref{eq:admmy-linearized} and $\beta> \max\{\gamma, 2L^2+L+2\}$, $(M_{\beta}^k)_{k\ge 0}$ is lower bounded and convergent;
                \end{enumerate}
                \begin{enumerate}
                \setcounter{enumi}{1}
                \item for Walkman with either \eqref{eq:admm3y'} or \eqref{eq:admmy-linearized}, the sequence $(x^k,\sy^k,\sz^k)_{k\ge 0}$ is bounded.
        \end{enumerate}
 
\end{lemma}
See the Appendix for a proof.
Based on Lemma \ref{lm:1}, we establish the convergence of subgradients of $ L_\beta^{k}$.
\begin{lemma}\label{lm:gk_convg}Take Assumptions \ref{ass-markov}--\ref{ass-prox} and Walkman with $\beta$ given in Lemma \ref{lm:1}.
%
For any given subsequence (including the whole sequence) with its index $(k_s)_{s\ge 0}$, { there exists a sequence $\{g^k\}_{k\ge 0}$ with $g^k\in\partial L_{\beta}^{k+1}$ containing} an almost surely convergent subsubsequence $(g^{k_{s_j}})_{j\ge 0}$, that is,
$$\Pr\Big(\lim_{j\to\infty}\|g^{k_{s_j}}\| = 0\Big)=1.$$
\end{lemma}
\begin{proof}
The proof sketch  is summarized as follows. 
        \begin{enumerate}

  \item We construct $g^k\in\partial L_{\beta}^{k+1}$ and show that its subvector $ q_i^{k}:=(g^k_x,g_{y_i}^k,g_{z_i}^k)$ satisfies $\lim_{k\to\infty}\EE\| q_{i_k}^{k-\uptau(\delta)-1}\|^2=0$, where the mixing time $\uptau(\delta)$ is defined in \eqref{eq:def_J}.
  \item  For $k\ge 0$, define the filtration of sigma algebras:
        \begin{align*}
        \chi^k\hspace{-0.8mm}:=\hspace{-0.8mm}\sigma\left(x^0,  \cdots, x^k, \sy^0, \cdots, \sy^k, \sz^0, \cdots, \sz^k, i_0, \cdots, i_k\right)\hspace{-0.8mm}.
        \end{align*}
We show that $$\EE\Big(\| q_{i_k}^{k-\uptau(\delta)-1}\|^2\Big|\chi^{k-\uptau(\delta)}\Big)\ge\hspace{-0.8mm} (1-\delta)\pi_*\|g^{k-\uptau(\delta)-1}\|^2,$$ where $\pi_*$ is the minimal value in the Markov chain's stationary distribution. From this bound and the result in step 1), we can get $\lim_{k\to\infty}\EE\|g^k\|=0$.
  \item From the result in the last step, we use some inequalities and the Borel-Cantelli lemma to obtain an almost surely convergent subsubsequence of $g^k$.
\end{enumerate}
 The details of these steps are given in the Appendix.
\end{proof}

\begin{theorem}\label{thm:convg} 
         Under Assumptions \ref{ass-markov}--\ref{ass-prox}, for $\beta\hspace{-0.8mm} > \hspace{-0.8mm}\max\{\gamma\hspace{-0.5mm}, \hspace{-0.5mm}2L+2\}$ (resp. $\beta\hspace{-0.8mm} > \hspace{-0.8mm}\max\{\gamma\hspace{-0.5mm}, \hspace{-0.5mm}2L^2+L+2\}$),  it holds that any limit point $(x^*,\sy^*,\sz^*)$  of  the sequence $(x^k,\sy^k,\sz^k)$ generated by Walkman with \eqref{eq:admm3y'} (resp. \eqref{eq:admmy-linearized}) satisfies: $x^*= y_i^*$, $i=1,\ldots,n$, where $x^*$ is a stationary point of \eqref{eq:mainprob}, with probability $1$,  that is,
        \begin{align}\label{eq:sta}
                \Pr\Big(0\in \partial r(x^*) + \frac{1}{n}\sum_{i=1}^n\nabla f_i(x^*)\Big)=1.
        \end{align}
        If the objective of \eqref{eq:mainprob} is convex, then $x^*$ is a minimizer.
\end{theorem}

\begin{proof}
By statement 2) of Lemma \ref{lm:1}, the sequence $(x^k,\sy^k,\sz^k)$ is bounded so there exists a convergent subsequence $(x^{k_s}, \sy^{k_s}, \sz^{k_s})$ converging to a limit point $(x^*,\sy^*,\sz^*)$ as $s\to \infty$. By continuity, we have
\begin{align}
L_\beta(x^*,\sy^*,\sz^*)= \lim_{s\to \infty} L_\beta(x^{k_s}, \sy^{k_s}, \sz^{k_s}).
\end{align}
Lemma \ref{lm:gk_convg} finds a subsubsequence $g^{k_{s_j}}\in \partial L_\beta^{k+1}$ such that  $\Pr\Big(\lim_{j\to\infty} \|g^{k_{s_j}}\| = 0\Big)=1$.
By the definition of \emph{general subgradient} (cf. \cite[Def. 8.3]{VarAnalysis}), we have $0\in\partial L_\beta(x^*, \sy^*, \sz^*)$.

This completes the proof of Theorem \ref{thm:convg}.
 \end{proof}

{\color{black}
Next, we derive the convergence rate for Walkman with a specific initialization, $z_i^0 = \nabla f_i(y_i^0)$. Specifically, to avoid consensus preprocessing, we need $\nabla f_i(y_i^0) = \beta y_i^0$. In other words, $y_i^0$ is a stationary point for the problem $\Min_{y\in\RR^p} f_i(y)-\frac{\beta}{2}\|y\|^2 $. This preprocessing can be accomplished without communication.
\begin{theorem}\label{thm:convg_rate_nonconvex}[Gradient sublinear convergence]
Under  Assumptions \ref{ass-markov}--\ref{ass-prox} and Walkman with $\beta$ given in Lemma \ref{lm:1}, and local variables initialized as $\nabla f_i(y_i^0) = \beta y_i^0=z_i^0, ~\forall i\in\{1,\cdots, n\}$, there exists a sequence $\{g^k\}_{k\ge 0}$ with $g^k\in\partial L_{\beta}^{k+1}$ satisfying 
\begin{align}
\min_{k\le K}\EE\|g^{k}\|^2\le \frac{\bar{C}}{K}( L_{\beta}^0-\underline{f}), ~\forall K> \uptau(\delta) +2,
\end{align}
where $\bar{C}$ is a constant merely depending on $\beta, L, \gamma$ and $n,  \uptau(\delta) $. With $\beta, L, \gamma$ independent from the network structure, one has $\bar{C}\sim O(\frac{\uptau(\delta)^2 + 1}{(1-\delta)n\pi_*})$, where $\uptau(\delta)$ is defined in \eqref{eq:def_J}.
\end{theorem}
\begin{proof}
The detailed proof can be found in Appendix \ref{appendixC}.
\end{proof}
It is possible, though more cumbersome, to show a sublinear convergence rate under a more general initialization. We decided not to pursue it.
}

\section{Linear convergence for Least Squares}\label{sc:ls}
In this section, we focus on the decentralize least-squares problem:
\begin{align}\label{ls-equi}
        \Min&\quad \frac{1}{n}\sum_{i=1}^{n}\frac{1}{2}\|\vA_i y_i - b_i\|^2, \nonumber \\
        \St& \quad y_1 = y_2 = \cdots = y_n = x,
\end{align}
which is a special case \eqref{eq:mainprob-equiv} with regularizer $r=0$, local objective $f_i(y_i):=\frac{1}{2}\|\vA_i y_i - b_i\|^2$ and gradient $\nabla f_i(y_i) = \vA_i\tran (\vA_iy_i - b_i)$. The Lipschitz constant $L$ in Assumption \ref{ass-lipschitz} equals ${\sigma^*_{\max}} :=\max_{i} \sigma_{\max}(\vA_i\tran\vA_i)$, where $\sigma_{\max}(\cdot)$ takes largest eigenvalue.
To assure that there exists a single optimum to problem \eqref{ls-equi}, the following analysis is based on the assumption that the matrix $\sum_{i=1}^n \vA_i\tran\vA_i$ is reversible, which implies   \eqref{ls-equi} is strongly convex.

We apply Walkman (or Algorithm 1) updating with $\prox_{f_i}$, i.e., utilizing \eqref{eq:admm3y'}, and starting from
\begin{align}
y_i^0 =& (\vA_i\tran\vA_i - \beta\vI)^{-1}(\vA_i\tran b_i), ~\forall i\in V, \label{eq:y_init}\\
z_i^0 =& \nabla f_i(y_i^0) = \vA_i\tran(\vA_iy_i^0 - b_i),~ \forall i\in V,\label{eq:z_init}
\end{align}
where \eqref{eq:y_init} is well defined for $\beta > \max_i \sigma_{\max}(\vA_i\tran \vA_i)$.
This is to ensure $y_i^0 - {z_i^0}/{\beta} = 0$ and thus \eqref{initialization} for all $k\ge 0$.

We analyze the complexities of Walkman for problem \eqref{ls-equi} based on the Lyapunov function $h_{\beta}(\sy):\RR^{np}\rightarrow\RR$,
\begin{align}\label{eq:hdef}
h_{\beta}(\sy) := &\  \frac{1}{n}\sum_{i=1}^n\big(\frac{\beta}{2}{\|y_i^k\|}^2-\frac{1}{2}{\|\vA_iy_i^k\|}^2+\frac{1}{2}{\|b_i\|}^2\big)\nonumber \\
&\quad -\frac{\beta}{2}{\|\vT\sy + c\|}^2,
\end{align}
where $\vT := \frac{1}{n}\hspace{-1mm}\big[(\vI\hspace{-0.5mm} -\hspace{-0.5mm} \frac{1}{\beta}\vA_1\tran \vA_1), \dots\hspace{-0.5mm}, (\vI \hspace{-0.5mm}-\hspace{-0.5mm}\frac{1}{\beta}\vA_n\tran \vA_n)\big]\hspace{-1mm} \in \hspace{-1mm} \RR^{p\times np}$ and $c :=  \frac{1}{n\beta}\sum_{i=1}^n \vA_i\tran b_i \in \RR^p$.
The following lemma relates $h_{\beta}(\sy)$ and the augmented Lagrangian sequence. 
\begin{lemma}
        With initialization \eqref{eq:y_init} and \eqref{eq:z_init}, it holds that
        \begin{align}\label{h=L}
                h_\beta(\sy^k) = L_{\beta}(x^{k+1},\sy^k;\sz^k).
        \end{align}
\end{lemma}
\begin{proof}
        From the optimality condition of \eqref{eq:admm3y'}, we can verify
        \begin{align}
        \vA_{i_k}\tran(\vA_{i_k}y_{i_k}^{k+1} - b_{i_k}) {=}\beta x^{k+1} +{z_{i_k}^k} -\beta y_{i_k}^{k+1} \overset{\rm(a)}{=} z_{i_k}^{k+1},\label{eq:z_y_kge1}
        \end{align}
for $k\ge 1$,
        where (a) follows from \eqref{eq:admm3z'}.
        In Walkman, each pair of $y_i$ and $z_i$ is either updated together, or both not updated.
        Then by applying \eqref{eq:z_init} and \eqref{eq:z_y_kge1}, we get
        \begin{align}
        z_i^{k} &= \vA_i\tran (\vA_iy_i^k - b_i), ~\forall i\in V,~k\ge 0. \label{eq:z_y}
        \end{align}
  Substituting \eqref{eq:z_y} into \eqref{eq-x-bar-update} and \eqref{eq:admm3x'} yields $x^{k+1} = \vT \sy^{k}+c, \quad \forall k\ge0.$
  Eliminating $z_i^k$ and $x^{k+1}$ in  $L_{\beta}(x^{k+1},\sy^k;\sz^k)$  using the above formulas produces \eqref{h=L}.
\end{proof}
The following lemma establishes that $h_{\beta}(\sy)$ is strongly convex and Lipschitz differentiable.
\begin{lemma}\label{lm:h_convex}
        For a network with $n\ge 2$ agents, and the parameter $\beta>{\sigma^*_{\max}}$, where ${\sigma^*_{\max}} :=\max_{i} \sigma_{\max}(\vA_i\tran\vA_i)$,  the function $h_{\beta}(\cdot)$ is
        \begin{enumerate}
                \item strongly convex with modulus $\nu=\frac{(n-1)(\beta - {\sigma^*_{\max}})}{n^2},$ and
                \item Lipschitz differentiable with Lipschitz constant $\bar{L}= \frac{\beta}{n}\big(1-\frac{1}{n}\big(1-\frac{\sigma^*_{\max}}{\beta}\big)^2\big).$
        \end{enumerate}
\end{lemma}
\begin{proof}
        As a quadratic function, $h_{\beta}(\cdot)$ is $\nu$-strongly convex with $\bar{L}$-Lipschitz gradients if, and only if, its Hessian (by \eqref{eq:hdef}) $ \vH$ satisfies
        \begin{align}\label{eq:Hbounds}
        \nu\vI\preceq\vH :=& \frac{\beta}{n}\vI_{np} -\frac{1}{n}\vA - \beta\vT\tran\vT\preceq  \bar{L}\vI,
        \end{align}
where $\vA :=  \diag\left(\vA_1\tran \vA_1, \vA_2\tran \vA_2, \cdots, \vA_n\tran \vA_n\right).$
        With $\beta> \max_i \sigma_{\max}(\vA_i\tran \vA_i)$, we define the symmetric positive definite matrices $\vD_i := \left(\vI -\frac{1}{\beta}\vA_i\tran \vA_i\right)^{1/2}$ for $i\in V$. The spectral norm of $\vD_i$ satisfies
        \begin{align}
        \Big(1-\frac{\sigma^*_{\max}}{\beta}\Big)^{\frac{1}{2}}\hspace{-2mm}\le\Big(1 -\frac{\sigma_{\max}(\vA_i\tran\vA_i)}{\beta}\Big)^{\frac{1}{2}}\le \|\vD_i\|\le 1.
        \end{align}
        Stacking $\vD_i$'s into
        \begin{align}
        \vD := \left[\begin{array}{c}\vD_1\\\vdots\\ \vD_n \end{array}\right].
        \end{align}
        Then, for any vector $\sw := {\rm col}\{w_1,\cdots,w_n\} \in \RR^{np}$ where $w_i\in \RR^p$, we have the interval bounds for $\|\diag(\vD)\sw\|$:
        \begin{align}
        \|\diag(\vD)\sw\| =&\left\|\left[\begin{array}{c}\vD_1w_1\\\vdots\\ \vD_n w_n \end{array}\right]\right\|\\
        \in&\Big[ \Big(1-\frac{\sigma^*_{\max}}{\beta}\Big)^{\frac{1}{2}}\|\sw\|, \|\sw\|\Big].
        \end{align}
        It is easy to check
        \begin{align}
          \sw\tran \vH\sw = \hspace{-1mm} \frac{\beta}{n}\hspace{-1mm}\left(\diag(\vD)\sw\right)\tran\hspace{-1mm}\left(\vI - \frac{1}{n}\vD\tran\vD\right)\hspace{-1mm}(\diag(\vD)\sw).
        \end{align}
        Therefore, we get \eqref{eq:Hbounds} from
        \begin{align}
         \sw\tran \vH\sw\geq & \frac{\beta}{n}\big(1-\frac{1}{n}\big)\|\diag(\vD)\sw\|^2 \\
         \geq&\underbrace{\frac{\beta}{n}\big(1-\frac{1}{n}\big)\big(1-\frac{\sigma^*_{\max}}{\beta}\big)}_{\nu}\|\sw\|^2
        \end{align}
        and
        \begin{align}
        \sw\tran \vH\sw\le& \frac{\beta}{n}\Big(\|\diag(\vD)\sw\|^2 -\frac{1}{n}\Big)\\
        \le& \frac{\beta}{n}\Big(\|\sw\|^2 -\frac{1}{n}\big(1-\frac{\sigma^*_{\max}}{\beta}\big)^2\|\sw\|^2\Big)\\
        =& \underbrace{\frac{\beta}{n}\big(1-\frac{1}{n}\big(1-\frac{\sigma^*_{\max}}{\beta}\big)^2\big)}_{\bar{L}}\|\sw\|^2.
        \end{align}
\end{proof}


\begin{lemma}\label{lm:h_opt}
With $\beta>{\sigma^*_{\max}}$, the unique minimizer of $h_{\beta}(\cdot)$ is $\sy^\star := {\rm col}\{y_1^\star, \cdots, y_n^\star\}$ with
$y_i^\star \equiv x^\star=(\sum_{i=1}^n\vA_i\tran\vA_i)^{-1}(\sum_{i=1}^n \vA_i\tran b_i)$. These components are also the unique solution to
\eqref{ls-equi}, as well as the unique minimizer of $\sum_{i=1}^{n}\frac{1}{2}\|\vA_i x - b_i\|^2$.
\end{lemma}
\begin{proof}
Since $\sy^*$ must satisfy $\nabla h_\beta(\sy^*)=0$, we have
\begin{align}
\nabla_{i} h_\beta(\sy^*)=&\frac{\beta}{n}(y_i^\star \hspace{-1mm}-\hspace{-1mm} \frac{1}{\beta}\vA_i\tran\vA_i y_i^\star )\hspace{-1mm}-\hspace{-1mm} \frac{\beta}{n}(\vI - \frac{1}{\beta}\vA_i\tran\vA_i)(\vT\sy^\star +c) \nonumber \\
 =& \frac{\beta}{n}(\vI - \frac{1}{\beta}\vA_i\tran\vA_i)(y_i^\star - \vT\sy^\star - c)=0.
\end{align}
Since $\vI - \frac{1}{\beta}\vA_i\tran\vA_i\succ 0$ with  $\beta>{\sigma^*_{\max}}$, we conclude
\begin{align}
y_i^\star - \vT\sy^\star - c =0, ~\forall i =1,\dots,n,
\end{align}
which implies $y^\star$ given in the Lemma. It is easy to verify the rest of the Lemma using optimality conditions.
\end{proof}
Define \textbf{one epoch} as $\uptau(\delta)$ iterations, and let
\begin{align}
        h_\beta^\star:= \Min_{\ssy}\{h_{\beta}(\sy)\},\quad F_t :=\EE h_{\beta}(\sy^{t\uptau(\delta)}) -  h_{\beta}^\star, \label{xcn3s0d}
\end{align}
{\color{black} where we use $t$ to index an epoch.}
The next lemma is  fundamental to the remaining analysis.

\begin{lemma}\label{lm:main_ineq}
        Under Assumption \ref{ass-markov} and {$\beta>2{\sigma^*_{\max}}+2$}, for any $\delta >0$, we have
        \begin{align}
        F_t^2\leq \frac{ 2\beta^2\uptau(\delta)}{n(1-\delta)\pi_*}\left(F_t - F_{t+1}\right)\cdot \EE\|\sy^{t\uptau(\delta)} - \sy^\star\|^2, \label{eq:main_ineq}
        \end{align}
        where $\uptau(\delta)$ is defined in \eqref{eq:def_J}.
\end{lemma}
\begin{proof}
We first upper bound $\|\nabla h_{\beta}(\sy^k)\|^2$. Verify
        \begin{align}\label{eq:gradihb}
        \nabla_{i} h_{\beta}(\sy^k) = 
\frac{\beta}{n}\vD_i^2\left(y_i^k - \vT\sy^k - c\right).
        \end{align}
Investigate step \eqref{eq:admm3y'} for $i=i_k$ as
        \begin{align}
        y_i^{k+1}  =& \argmin_{y} \frac{1}{2}\|\vA_iy - b_i\|^2+\frac{\beta}{2}\|y -x^{k+1} -\frac{1}{\beta}z_i^k\|^2 \nonumber \\
        =& (\vA_i\tran \vA_i +\beta\vI)^{-1}\left(\vA_i\tran b_i +\beta x^{k+1}+z_i^k\right) \nonumber \\
        {\overset{\rm (a)}{=}} &  (\vA_i\tran \vA_i +\beta\vI)^{-1}\left(\beta\vT\sy^k +\beta c +\vA_i\tran\vA_iy_i^k\right) \nonumber \\
        = & y_i^k + (\vI + \frac{1}{\beta}\vA_i\tran \vA_i)^{-1}\left(\vT\sy^k + c - y_i^k\right)\nonumber \\
        \overset{\rm \eqref{eq:gradihb}}{=} & y_i^k - \frac{n}{\beta}(\vI + \frac{1}{\beta}\vA_i\tran \vA_i)^{-1}\vD_i^{-2}\left(\nabla_i h_{\beta}(\sy^k) \right),\label{eq:y_i_update}
        \end{align}
        where (a) follows from \eqref{eq:z_y} and $\vT$'s definition. 
        Thence, 
        \begin{align}
        \|\nabla_{i_k} h_{\beta}(\sy^k) \| =& \frac{\beta}{n}\Big\|(\vI - \frac{1}{\beta^2}(\vA_{i_k}\tran\vA_{i_k})^2)(y_{i_k}^{k+1} - y_{i_k}^k)\Big\| \nonumber\\
        \leq& \frac{\beta}{n}\|\sy^{k+1} - \sy^k\|,\label{xnwesd}
        \end{align}
        For any $k\geq\uptau(\delta)-1$, we further have 
        \begin{align}
        &\|\nabla_{i_k} h_{\beta}(\sy^{k-\uptau(\delta)+1})\|^2 \nonumber \\
        =& \|\nabla_{i_k} h_{\beta}(\sy^{k-\uptau(\delta)+1}) - \nabla_{i_k} h_{\beta}(\sy^{k}) +\nabla_{i_k} h_{\beta}(\sy^{k})\|^2 \nonumber \\
        \leq& 2\|\nabla_{i_k} h_{\beta}(\sy^{k-\uptau(\delta)+1}) - \nabla_{i_k} h_{\beta}(\sy^{k})\|^2 + 2\|\nabla_{i_k} h_{\beta}(\sy^{k})\|^2 \nonumber \\
        \overset{\eqref{xnwesd}}{\leq}& \frac{2\beta^2\uptau'}{n^2} \hspace{-3mm}\hspace{-1mm} \sum_{d = k-\uptau(\delta)+1}^{k-1}\hspace{-1mm}\|\sy^{d+1} \hspace{-1mm}-\hspace{-1mm} \sy^d\|^2 \hspace{-1mm}+\hspace{-1mm} \frac{2\beta^2}{n^2}\|\sy^{k+1} \hspace{-1mm}-\hspace{-1mm} \sy^k\|^2 \nonumber \\
        \leq& \max\Big\{\frac{2\beta^2\uptau'}{n^2}, \frac{2\beta^2}{n^2}\Big\}\sum_{d = k-\uptau(\delta)+1}^{k}\|\sy^{d+1} \hspace{-1mm}-\hspace{-1mm} \sy^d\|^2 \nonumber \\
        \leq& \frac{2\beta^2\uptau(\delta)}{n^2}\sum_{d = k-\uptau(\delta)+1}^{k}\|\sy^{d+1} - \sy^d\|^2,\label{eq:nabla_ik}
        \end{align}
        where $\uptau'= \uptau(\delta)-1$ and the last inequality holds because $\beta > {\sigma^*_{\max}}$. With the filtration
$
                \cX^k = 
\sigma\{\sy^0,\cdots,\sy^k, i_0,\cdots, i_{k-1}\},
$
        \begin{align}
        &\hspace{-1cm}\EE\left(\|\nabla_{i_k} h_{\beta}(\sy^{k-\uptau(\delta)+1})\|^2|\chi^{k-\uptau(\delta)+1}\right) \nonumber \\
        &=\EE\left(\|\nabla_{i_k} h_{\beta}(\sy^{k-\uptau(\delta)+1})\|^2|\sy^{k-\uptau(\delta)+1}, i_{k-\uptau(\delta)}\right) \nonumber \\
        &= \sum_{j=1}^N [\vP^{\uptau(\delta)}]_{i_{k-\uptau(\delta)}, j}\|\nabla_{j} h_{\beta}(\sy^{k-\uptau(\delta)+1})\|^2 \nonumber \\
        &\overset{\eqref{eq:mix_time_prop}}{\ge} (1-\delta)\pi_*\|\nabla h_{\beta}(\sy^{k-\uptau(\delta)+1})\|^2.\label{eq:cond_exp}
        \end{align}
        Reverting the sides of \eqref{eq:cond_exp} and taking expectation over $\cX^{k-\uptau(\delta)+1}$, followed by applying \eqref{eq:nabla_ik}, we have for $k \ge \uptau(\delta)-1$
        \begin{align}
        &\EE\|\nabla h_{\beta}(\sy^{k-\uptau(\delta)+1})\|^2 \nonumber \\
        &\leq \frac{ 2\beta^2\uptau(\delta)}{n^2(1-\delta)\pi_*}\hspace{-1.5mm}\sum_{d = k-\uptau(\delta)+1}^{k}\EE\left(\|\sy^{d+1} - \sy^d\|^2\right). \label{eq:nabla_func}
        \end{align}
        Notice that
        \begin{align}\label{xjwens9}
                &\ h_\beta(\sy^k) - h_\beta(\sy^{k+1}) \nonumber \\
                \overset{\eqref{h=L}}{=}&\ L_\beta(x^{k+1},\sy^k;\sz^k) - L_\beta(x^{k+2},\sy^{k+1};\sz^{k+1})\nonumber \\
                =&\ L_\beta(x^{k+1},\sy^k;\sz^k) \hspace{-1mm}-\hspace{-1mm} L_{\beta}^{k+1} \hspace{-1mm}+\hspace{-1mm} L_{\beta}^{k+1} \hspace{-1mm}-\hspace{-1mm}  L_\beta(x^{k+2},\sy^{k+1};\sz^{k+1}) \nonumber \\
                \ge&\ \frac{1}{n}\|\sy^{k} - \sy^{k+1}\|^2,
        \end{align}
        where the last line follows from parts 1 and 2 of Lemma \ref{lm:1}. Combining \eqref{xjwens9} and \eqref{eq:nabla_func}, we get
        \begin{align}\hspace{-2mm}
                &\EE\|\nabla h_{\beta}(\sy^{k-\uptau(\delta)+1})\|^2 \nonumber \\
                &\leq  \hspace{-1mm}  \frac{ 2\beta^2\uptau(\delta)}{n(1-\delta)\pi_*}\, \EE\left(h_\beta(\sy^{k-\uptau(\delta)+1}) \hspace{-0.5mm}-\hspace{-0.5mm} h_\beta(\sy^{k+1})\right).\label{eq:gradient_bound}
        \end{align}
        Now with $k=(t+1)\uptau(\delta)-1$, 
  \eqref{eq:gradient_bound} reduces to
        \begin{align}
        &\EE\|\nabla h_{\beta}(\sy^{t\uptau(\delta)})\|^2
        \nonumber \\
        &\le\frac{ 2\beta^2\uptau(\delta)}{n(1-\delta)\pi_*}\, \EE\left(h_\beta(\sy^{t\uptau(\delta)}) \hspace{-0.5mm}-\hspace{-0.5mm} h_\beta(\sy^{(t+1)\uptau(\delta)})\right) \nonumber \\
        &\overset{\eqref{xcn3s0d}}{=}\  \frac{ 2\beta^2\uptau(\delta)}{n(1-\delta)\pi_*}\, (F_t - F_{t+1})  \label{eq:gradient_bound-2}
        \end{align}
        By the convexity of $h_{\beta}(\cdot)$,
        \begin{align}
        \hspace{-2mm}\EE h_{\beta}(\sy^{t\uptau(\delta)}) \hspace{-1mm}-\hspace{-1mm} h_\beta^\star \leq \EE\langle\nabla h_{\beta}(\sy^{t\uptau(\delta)}), \sy^{t\uptau(\delta)} \hspace{-1mm}-\hspace{-1mm} \sy^\star\rangle. \label{eq:h_convex}
        \end{align}
        Since both sides of \eqref{eq:h_convex} are nonnegative, we square them and use the Cauchy-Schwarz inequality to get
        \begin{align}
        F_t^2\leq \EE\|\nabla h_{\beta}(\sy^{t\uptau(\delta)})\|^2\cdot \EE\|\sy^{t\uptau(\delta)} - \sy^\star\|^2. \label{eq:Eh_convex}
        \end{align}
        Substituting \eqref{eq:gradient_bound} into \eqref{eq:Eh_convex} completes the proof.
\end{proof}

Now we are ready to establish the linear convergence rate of the sequence $(F_t)_{t\geq 0}$.
\begin{theorem}\label{thm:linear_convg}
        Under Assumption \ref{ass-markov}, for $\beta>2{\sigma^*_{\max}}+2$, we have linear convergence (with $\nu$ given in Lemma \ref{lm:h_convex}):
        \begin{align}\label{xc;es}
                F_{t+1} \le
                \left( 1+ \frac{n(1-\delta)\pi_*\nu}{{ 4\beta^2\uptau(\delta)}}\right)^{-1}
                F_t, \quad \forall\ t\geq 0.
        \end{align}
\end{theorem}
\begin{proof}
        By the strong convexity of $h_{\beta}(\cdot)$ and $\sy^\star= \arg\Min h_{\beta}(\sy)$, it holds for any $\sy\in \RR^{np}$ that,
        \begin{align}\label{xn38}
        \frac{\nu}{2}\|\sy - \sy^\star\|^2 \le h_{\beta}(\sy) - h_{\beta}(\sy^\star).
        \end{align}
        Hence,
        \begin{align}
        \EE\|\sy^{t\uptau(\delta)} - \sy^\star\|^2\leq \frac{2(\EE h_{\beta}(\sy^{t\uptau(\delta)}) - h_{\beta}^\star) }{\nu} = \frac{2 F_t}{\nu}.\label{eq:to_optimal_bound}
        \end{align}
        Substituting \eqref{eq:to_optimal_bound} into \eqref{eq:main_ineq}, we have
        \begin{align}\label{xcn239sdu}
        F_t^2\leq \frac{C}{\nu}\left(F_t - F_{t+1}\right)F_t,\quad \mbox{where} \quad C = \frac{ 4\beta^2\uptau(\delta)}{n(1-\delta)\pi_*}.
        \end{align}
        By \eqref{xjwens9}, the sequence $\{h_\beta(\sy^k)\}$ is non-increasing, implying $0 \le F_{t+1} \le F_t$. This together with \eqref{xcn239sdu} yields
        \begin{align}
        F_tF_{t+1}\leq& \frac{C}{\nu}\left(F_t - F_{t+1}\right)F_t,
        \end{align}
        which is equivalent to \eqref{xc;es}.
\end{proof}

Theorem \ref{thm:linear_convg} states that Walkman for decentralized least squares converges linearly by epoch (every $\uptau(\delta)$ iterations).

\section{Communication Analysis}\label{sc:comm}
This section derives and compares communication complexities with some state-of-the-art methods to solve problem \eqref{admm-comm} {\color{black} in solving two different types of problems: the decentralized least squares problem and  the general nonconvex nonsmooth problem}. %
{\color{black} In the following analysis, communication of $p$-dimensional variables between a pair of agents is taken as $1$ unit of communication, while $\nu$ and $L$ are taken as constants independent from network scale $n$.}

\subsection{Solving Least Squares Problem}\label{sec:comm_ls}
First, we establish the communication complexity of Walkman. From \eqref{xn38} and \eqref{xcn3s0d}, we have
\begin{align}
        \EE \|\sy^{t\uptau(\delta)} - \sy^\star\|^2 \le \frac{2}{\nu} F_t \overset{\rm \eqref{xc;es}}{\le} \left(\frac{2}{\nu}\right)
        \left( 1+ \frac{n(1-\delta)\pi_*\nu}{{ 4\beta^2\uptau(\delta)}}\right)^{-t} F_0.
\end{align}
To achieve   mean-square deviation  $G_t := \EE \|\sy^{t\uptau(\delta)} - \sy^\star\|^2 \le \epsilon$,
it is enough to have
\begin{align}
        \left(\frac{2F_0}{\nu}\right)
        \left( 1+ \frac{n(1-\delta)\pi_*\nu}{{ 4\beta^2\uptau(\delta)}}\right)^{-t} \le \epsilon,
\end{align}
which is implied by
\begin{align}\label{xnwesdf}
        t = \ln\left(\frac{2F_0}{\nu \epsilon}\right)\Big/\ln\left(1 + \frac{n(1-\delta)\pi_*\nu}{{ 4\beta^2\uptau(\delta)}}\right).
\end{align}
Since $\beta$ can be regarded as constants that are independent of network size $n$, and $\nu$ is $O(\frac{1}{n})$, 
we can write: 
\begin{align}
        t \sim O\Big( \ln\big( \frac{n}{\epsilon} \big)/\ln \big( 1 + \frac{(1-\delta)\pi_*}{\uptau(\delta)}\big) \Big)
\end{align}
For each epoch $t$, there are $\uptau(\delta)$ iterations, which use $O(\uptau(\delta))$ communication. Hence, to guarantee $G_t\le \epsilon$, the total communication complexity is
\begin{align}
        O\bigg( \Big( \ln\big( \frac{n}{\epsilon} \big)/\ln \big( 1 +  \frac{(1-\delta)\pi_*}{\uptau(\delta)}\big) \Big) \cdot \uptau(\delta) \bigg) \label{wadmm-comm}
\end{align}
Recall the definition of $\uptau(\delta)$ in \eqref{eq:def_J}, by setting $\delta$ as $1/2$, the communication complexity is
\begin{align}
        O\bigg( \Big( \underbrace{\ln\big( \frac{n}{\epsilon} \big)/\ln \big( 1 +  \frac{(1-\sigma(\vP))\pi_*}{2\ln\frac{2}{ \pi_*}}\big)}_{\mbox{epoch number}}\Big) \cdot \underbrace{\frac{1}{1-\sigma(\vP)}\ln\frac{2}{\pi_*}}_{\mbox{comm. per epoch}}\bigg), \label{wadmm-comm}
\end{align}
where we remember  $\sigma(\vP) := \sup_{f\in\RR^n:f\tran \vone = 0}\| f\tran \vP\|/\|f\|.$

For simplicity of expression and comparison, in the succeeding parts we assume that the Markov chain is reversible with $\vP = \vP\tran$ and
it admits a uniform stationary distribution $\pi^T = \pi^T \vP$ with:
$$\pi=[1/n, \cdots, 1/n]\tran \in \RR^n,$$ which implies $\pi_*=\Min_i\pi_i=1/n.$
 With $\vP$ being a symmetric real matrix, we also have $\sigma(\vP)=\lambda_2(\vP)=\max\{|\lambda_i(\vP)|:\lambda_i(\vP)\not=1\}$.
 We get the total communication complexity of Walkman as
\begin{align}
\hspace{-2mm}
O\bigg( \underbrace{\Big(\ln\big(\frac{n}{\epsilon}\big) \Big/ \ln\big(1 + \frac{1-\lambda_2(\vP)}{2n\ln(2n)}\big)\Big)}_{\mbox{epoch numbers}}\hspace{-0.5mm} \cdot  \hspace{-2.5mm}\underbrace{\Big(\frac{\ln(n)}{1-\lambda_2(\vP)}\Big)}_{\mbox{comm. per epoch}}\hspace{-2mm} \bigg)\label{walk-admm-comm}
\end{align}

\subsubsection{Communication comparisons}\label{sec:comm_comp_ls}
For comparison, we list the communication complexities of some existing algorithms.
{\color{black}

Firstly, we study the communication complexity of ESDACD \cite{hendrikx2018accelerated}, which is an accelerated generalization to the randomized gossip method originally designed to solve the average consensus problem. When applying ESDACD, the agents should be synchronized to keep track of which iteration the network is going through. And in each iteration, only one edge in the network is activated to communicate bi-directionally. 
Based on Theorem 1 and  following the derivation of Section B2 of \cite{hendrikx2018accelerated}, with algorithmic parameter $\mu_{i,j} =  1$ and all the edges uniformly selected with probability $1/m$\footnote{Nonuniform selection of edges is not practical in real applications. Since each agent should generate the randomly selected edge in each iteration with the same seed, nonuniform selection of edges implies each agent should cache the diverse sampling probabilities for each edge.}, the communication complexity of ESDACD to achieve the deviation $G_t\le \epsilon$ is 
\begin{align}\label{esdacd-comm}
O\Big(\ln\left(\frac{1}{\epsilon}\right)\cdot \frac{m}{\sqrt{d_{\min}(1-\lambda_2(\vP))}}\Big),
\end{align}
where $d_{\min}$ denotes the smallest degree among $d_1,\cdots, d_n$.

As for RW-ADMM\cite{shah2018linearly}, in each iteration, it evokes the communications between the activated agent and all its neighbors, and thus consumes $d_{\rm ave}$ communications per iteration  where $d_{\rm ave}=\frac{1}{n}\sum_{i=1}^{n}d_i$ is the averaged degree in the network. This implies that RW-ADMM requires more communications than Walkman per iteration. To calculate the communication complexity of RW-ADMM, we consider a simple $d$-regular graph in which $m = nd/2$. We also assume the state transition matrix $\vP$ is symmetric and doubly stochastic so that $\pi_{\max}=\pi_{\min}=1/n$. In addition, we assume the current agent will activate one of its $d$ neighbors with a uniform probability $p = 1/d$ and thus it holds that $p_{\max}=p_{\min}=1/d$. Under these conditions, the communication complexity for RW-ADMM is verified as 
\begin{align}\label{rw-admm-communication}
O\Big(\ln\left(\frac{1}{\epsilon}\right)\cdot \frac{m d}{\sqrt{(1-\lambda_2(\vP))}}\Big).
\end{align}
Next, we  consider gossip based methods. 
}
D-ADMM\cite{shi2014linear} has: 
\begin{align}
\hspace{-2.2mm}
O\bigg( \Big( \underbrace{\ln\big(\frac{1}{\epsilon}\big) \Big/ \ln\big(1 + \sqrt{1 -\lambda_2(\vP)}\big)}_{\mbox{iteration numbers}} \Big) \cdot  \hspace{-2.5mm}\underbrace{m}_{\mbox{comm. per iter.}} \hspace{-1mm}\bigg) \label{admm-comm}
\end{align}
where $m$ is the number of edges. 
The communication complexity of EXTRA \cite{shi2015extra} is
\begin{align} O\bigg( {\Big(\ln\big(\frac{1}{\epsilon}\big) \Big/ \ln\big(2 - \lambda_2(\vP)\big)\Big)} \cdot  m \bigg) \label{extra-comm}
\end{align}
As to exact diffusion\cite{yuan2017exact1}, the communication complexity is
\begin{align} O\bigg( {\Big(\ln\big(\frac{1}{\epsilon}\big) \Big/ \ln\big(1 + \frac{1 - \lambda_2(\vP)}{\lambda_2(\vP) + C}\big)\Big)} \cdot  m \bigg) \label{ed-comm}
\end{align}
where $C$ only depends on the condition number of the objective function, independent of $\lambda_2(\vP)$ and $n$.

{\color{black}
Considering the case $\epsilon\le 1/{\rm e}$, it holds $\ln( {n}/{\epsilon})\leq \ln n\cdot \ln(1/\epsilon)$.
Since 
$\ln(1+x) \approx x$ for $x$ close to $0$, Walkman in \eqref{walk-admm-comm} can be simplified to:
\begin{align}
        O\bigg( {\ln\big(\frac{1}{\epsilon}\big)} \cdot  {\frac{n\ln^3(n)}{(1-\lambda_2(\vP))^2}}\bigg). \label{walk-admm-comm1}
\end{align}
We similarly simplify the communication complexities in \eqref{admm-comm}, \eqref{extra-comm}, and \eqref{ed-comm}. They are listed in Table I in \S\ref{sec-intro-contribution}. 
{\color{black} Clearly, ESDACD has a better communication complexity than all the compared methods but may still be worse than Walkman.}

{\color{black} 
 Walkman is more communication efficient than ESDACD when
\begin{align}\label{eq:comm_efficient_cond1}
      {\frac{n\ln^3(n)}{(1-\lambda_2(\vP))^2}} \le \frac{m}{(d_{\min}(1-\lambda_2(\vP)))^{1/2}}.
\end{align}
With $d_{\min}\le m/n$, a sufficient condition for \eqref{eq:comm_efficient_cond1} is
\begin{align}\label{condition-walk-ADMM}
        \lambda_2(\vP) \le 1 - \frac{n^{1/3}[\ln(n)]^{2}}{m^{1/3}} \approx 1 - \big(\frac{n}{m}\big)^{1/3},
\end{align}
where the approximation holds for $\ln(n)\ll n$ and with $\ln(n)$ ignored. 
Condition \eqref{condition-walk-ADMM} indicates the network has moderately good connectivity. When this holds, Walkman exhibits superior communication efficiency than every compared algorithm.}}

\subsubsection{Communication for different graphs}\label{sec:graph_exmp_ls}

Let us consider three classes of graphs for concrete communication complexities. 

\noindent \textbf{Example 1 (Complete graph)} In a complete graph, every agent connects with all the other nodes. The number of edges $m = O(n^2)$ and $\lambda_2(\vP) = 0$, $d_{\min}=n$. Consequently, the communication complexity of Walkman is $O(\ln(1/\epsilon)n\ln^3(n))$ while {\color{black} that of ESDACD is $O(\ln(1/\epsilon)n^{3/2})$}, and those of the other algorithms are $O(\ln(1/\epsilon)n^2)$. Noticing  $\ln^3(n)\ll n^{1/2}$,Walkman is 
more communication efficient. 

\noindent \textbf{Example 2 (Random graph)} Consider the random graphs by Edgar Gilbert \cite{gilbert1959random}, $G(n, p)$, in which an $n$-node graph is generated with each edge populating independently with probability $p\in(0,1)$.
Let $\vA\in\RR^{n\times n}$ denote the adjacency matrix of the generated graph, with $A_{i,j}=1$ if nodes $i$ and $j$ are connected, and  $0$ otherwise.
The $(i,j)$-th $(i\neq j)$, entry of the  transition probability matrix $\vP$ is $P_{i,j} = \frac{A_{i,j}}{{d_{\max}}}$, where $d_{\max} = \max_{i} \sum_{j=1}^n A_{i,j}$ is the maximal degree of all the $n$ nodes. For any $i$, the diagonal entry is $P_{i,i} = 1 -\frac{\sum_{j\neq i}A_{i,j}}{{d_{\max}}}$. 
It can be shown
$\EE [m] = \frac{p(n^2-n)}{2}= O(n^2)$.
By union bound and Bernstein's inequality, one can easily derive $d_{\max}$ concentrates around $(n-1)p$. 
Further by Theorem 1 of \cite{chung2011spectra}, $1 - \lambda_2(\vP)$ concentrates around $1 - \frac{\lambda_{2}(\bar{\vA})}{(n-1)p}$, where \begin{align}
\bar{A}_{i,j}=\begin{cases}p & i\neq j\\ 0 & i= j.\end{cases}
\end{align}
Since $\bar{\vA}$ is a Toeplitz matrix, one can verify that $\lambda_{2}(\bar{\vA}) = p$, that is,
\begin{align}
1 - \lambda_2(\vP) = O\left(1\right).
\end{align}
With such setting, Walkman a communication complexity of roughly $O(\ln(1/\epsilon)n\ln^3 n )$ while {\color{black} that of ESDACD is $O(\ln(1/\epsilon)n^{3/2})$}, and the other algorithms have $O(\ln(1/\epsilon)n^2 )$. 
Hence, Walkman is more communication-efficient when $n$ is sufficiently large. 

\noindent \textbf{Example 3 (Cycle graph)} Consider a cycle, where each agent connects with its previous and next neighbors. One can verify that
\begin{align}
        1 - \lambda_2(\vP) =  O\left( 1 - \cos(2\pi/n)\right) = O\left(1/n^2\right),
\end{align}
and $m = O(n)$. Hence, Walkman has a communication complexity of roughly $O( \ln(1/\epsilon)n^5\ln^3 n)$ while, in \eqref{admm-comm} {\color{black} and \eqref{esdacd-comm}, D-ADMM and ESDACD have $O(n^2 \ln(1/\epsilon))$}, and in \eqref{ed-comm}
--\eqref{ed-comm}, EXTRA and exact diffusion have $O(n^3 \ln(1/\epsilon))$, so Walkman is less  communication-efficient.

%
%
%
%
%

{\color{black}\subsection{Solving General Nonconvex Nonsmooth Problems}
According to Theorem \ref{thm:convg_rate_nonconvex}, we first derive the communication complexity of Walkman. 
To achieve the ergodic gradient deviation $E_t:= \min_{k\le t} \EE\|g^k\|^2\le \epsilon$ for any $t>\uptau(\delta)+2$, it is sufficient to have
\begin{align}\label{eq:comm_ncvx_1}
\frac{\bar{C}}{t}\left(L_{\beta}^0 -\underline{f}\right)\le \epsilon.
\end{align}
Taking $L_{\beta}^0$ and $\underline{f}$ as constant independent from $n$ and the network structure, one has
\begin{align}
t\sim O\Big(\frac{1}{\epsilon}\cdot \frac{\uptau(\delta)^2 + 1}{(1-\delta)n\pi_*}\Big)
\end{align}

Recall the definition of $\uptau(\delta)$ in \eqref{eq:def_J}, by setting $\delta$ as $1/2$, the communication complexity is
\begin{align}
O\Big(\frac{1}{\epsilon}\cdot \frac{\ln^2\big(\frac{1}{\pi_*}\big)}{n\pi_*(1-\sigma(\vP))^2}\Big).
\end{align}

We consider a reversible Markov chain with $\vP\tran=\vP$ embedded on an undirected graph, and have the communication complexity of Walkman is 
\begin{align}
O\Big(\frac{1}{\epsilon}\cdot \frac{\ln^2n}{(1-\lambda_2(\vP))^2}\Big).
\end{align}

\subsubsection{Communication comparisons on different graphs}
Next, we compare the communication complexity of Walkman with existing algorithms, D-GPDA  \cite{sun2018distributed}  and xFILTER  \cite{sun2018distributed}  on two specific types of graph structures.
On a complete graph, the communication complexity of Walkman is $O\Big(\frac{\ln^2 n}{\epsilon}\Big)$, whereas, according to  \cite{sun2018distributed} ,  the better communication complexity between D-GPDA and xFILTER is $O\Big(\frac{n^2}{\epsilon}\Big)$.
Next, we consider the cycle graph, which is sparsely connected. 
Walkman consumes $O\Big(n^4\frac{\ln^2 n}{\epsilon}\Big)$ amount of communication on it, whereas  the better communication complexity between D-GPDA and xFILTER is $O\Big(\frac{n^2}{\epsilon}\Big)$. 
Hence, we can draw a similar conclusion as in Section \ref{sec:comm_ls}, that is, Walkman  is more communication efficient  on a more densely connected graph. 
}

\section{Numerical Experiments}\label{sc:num}
In this section, we compare Walkman with existing state-of-the-art decentralized methods through numerical experiments. Consider a network of $50$ nodes that are randomly placed in a $30\times 30$ square. Any two nodes within a distance of $15$ are connected; others are not. We set the probability transition matrix $\vP$ as $[\vP]_{ij} = 1/d_i$.
{\color{black}Algorithmic parameters in the following experiments are set as follows. For the random-walk (RW) incremental algorithm, we have used both a fixed step-size of 0.001 and a sequence of decaying step-sizes $\Min\{0.01, 5/k\}$.  For other algorithms, we have hand-optimized their parameters by grid-search}. 

{\color{black}\subsection{Decentralized least squares}}
The first experiment uses least squares in  \eqref{ls-equi} with $\vA_i \in \RR^{5\times 10}$, $x\in \RR^{10}$ and $b_i\in \RR^5$. Each entry in $\vA_i$ is generated from the standard Gaussian distribution, and $b_i := \vA_i x_0 + v_i$, where $x_0 \sim \cN(0, I_{10})$ and $v_i \sim \cN(0, 0.1\times I_5)$. Fig. \ref{fig:least_squares} compares different algorithms. {\color{black} In this experiment, the comparison methods include the randomized-gossip type method (ESDACD), those with dense communications (D-ADMM, EXTRA, exact diffusion), DIGing over time-varying graph, RW Incremental method and RW-ADMM with mixed communication pattern. To be noted, we implement all these methods in the synchronous fashion, i.e., an iteration would not start before a priori iteration completes. As for a method over time-varying graph, DIGing is conducted with merely one edge uniformly randomly chosen in each time instance. For ESDACD, the activated edge in each iteration is also drawn independently from a uniform distribution.} 

In the left plot of Fig. \ref{fig:least_squares}, we count one communication for each transmission of a $p$-length vector ($p=10$ is the dimension of $x$). It is observed that {\color{black} Walkman with \eqref{eq:admm3y'} is much more communication efficient than the other algorithms, while Walkman with \eqref{eq:admmy-linearized} is comparable to ESDACD and DIGing}. In the right plot of Fig. \ref{fig:least_squares}, we illustrate the running times of these methods.

While a running time should in general include the times of computing, communication, and other overheads, we only include communication time and allows simultaneous communication over multiple edges for non-incremental algorithms. However, we assume each communication follows an i.i.d. exponential distribution with parameter $1$. Each iteration of D-ADMM, EXTRA, and exact diffusion waits for the completion of the slowest communication (out of $2m$ communications), which determines the communication time of that iteration. In contrast, {\color{black} ESDACD, DIGing}, random-walk incremental algorithms and Walkman only use one communication per iteration. {\color{black} The communication time per iteration of RW-ADMM is in between, as it waits for the slowest communication in the neighborhood to complete.} Under our setting, Walkman takes longer to converge than D-ADMM, EXTRA, and exact diffusion. {\color{black} It is observed that Walkman with \eqref{eq:admm3y'} outperforms RW-ADMM and ESDACD in both communication cost and running time. In addition, D-ADMM is also observed more efficient than RW-ADMM, which is consistent with the communication complexity we derived in \eqref{rw-admm-communication}.}

\begin{figure}[h]
        \centering
        \includegraphics[width= 0.5\textwidth]{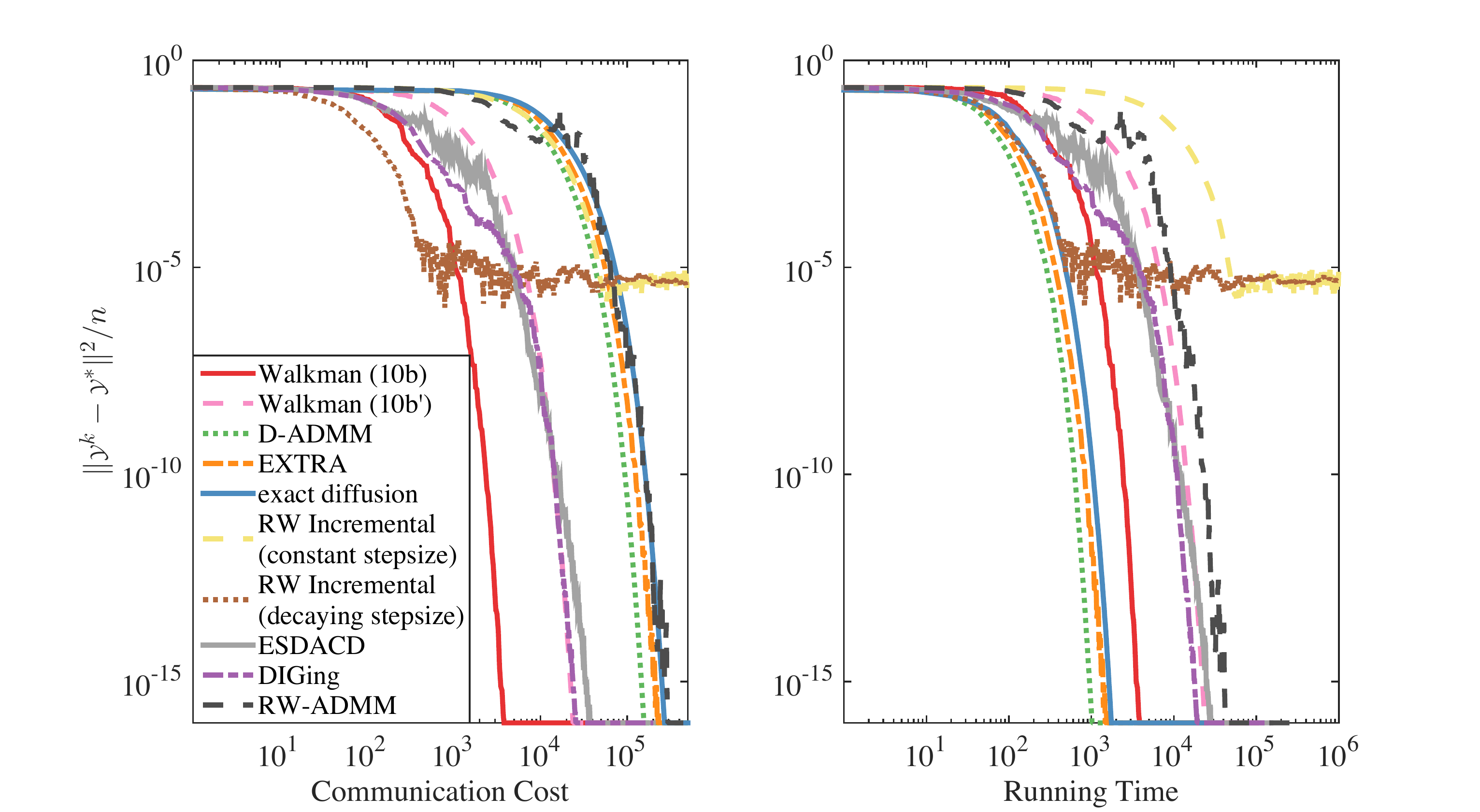}
        \caption{\footnotesize Performance of decentralized algorithms on least squares.}
        \label{fig:least_squares}
\end{figure}

{\color{black}\subsection{Decentralized sparse logistic regression}}
The second experiment solves the logistic regression problem
\begin{equation}\label{eq:logreg_sparse}
\Min_{x\in\RR^p} ~\lambda \|x\|_1\hspace{-1mm}+\hspace{-1mm}\frac{1}{nb}\sum_{i=1}^n\sum_{j=1}^b \log\left(1\hspace{-0.5mm}+\hspace{-0.5mm}\exp(-y_{ij}v_{ij}\tran x)\right) ,
\end{equation}
where $y_{ij}\in\{-1, 1\}$ denotes the label of the $j$th sample kept by the $i$th agent, and $v_{ij}\in \RR^p$ represents its feature vector, and there are $b$ samples kept by each agent. In this experiment, we set $b = 10, p =5$. Each sample feature $v_{ij}\sim\cN(0, 1)$.
To generate $y_{ij}$, we first generate a random vector $x^0\in\RR^5\sim\cN(0, I)$.
Then we generate a uniformly distributed variable $z_{ij}\sim\cU(0,1)$, and if $z_{ij}\leq 1/[1+\exp(-v_{ij}\tran x^0)]$, $y_{ij}$ is taken as $1$; otherwise $y_{ij}$ is set as $-1$. We run the simulation over the same network as the above least-square problem.
Due to the nonsmooth term in \eqref{eq:logreg_sparse}, EXTRA and exact diffusion is not applicable in this problem.
Instead, we compare Walkman with PG-EXTRA \cite{shi2015proximal}, {\color{black} D-ADMM} and random walk proximal gradient method, which conducts one-step proximal gradient operation when an agent receives the variable $x$. 

The communication efficiency of Walkman is also observed in Fig. \ref{fig:logistic-regression}.

\begin{figure}[h]
        \centering
        \includegraphics[width= 0.5\textwidth]{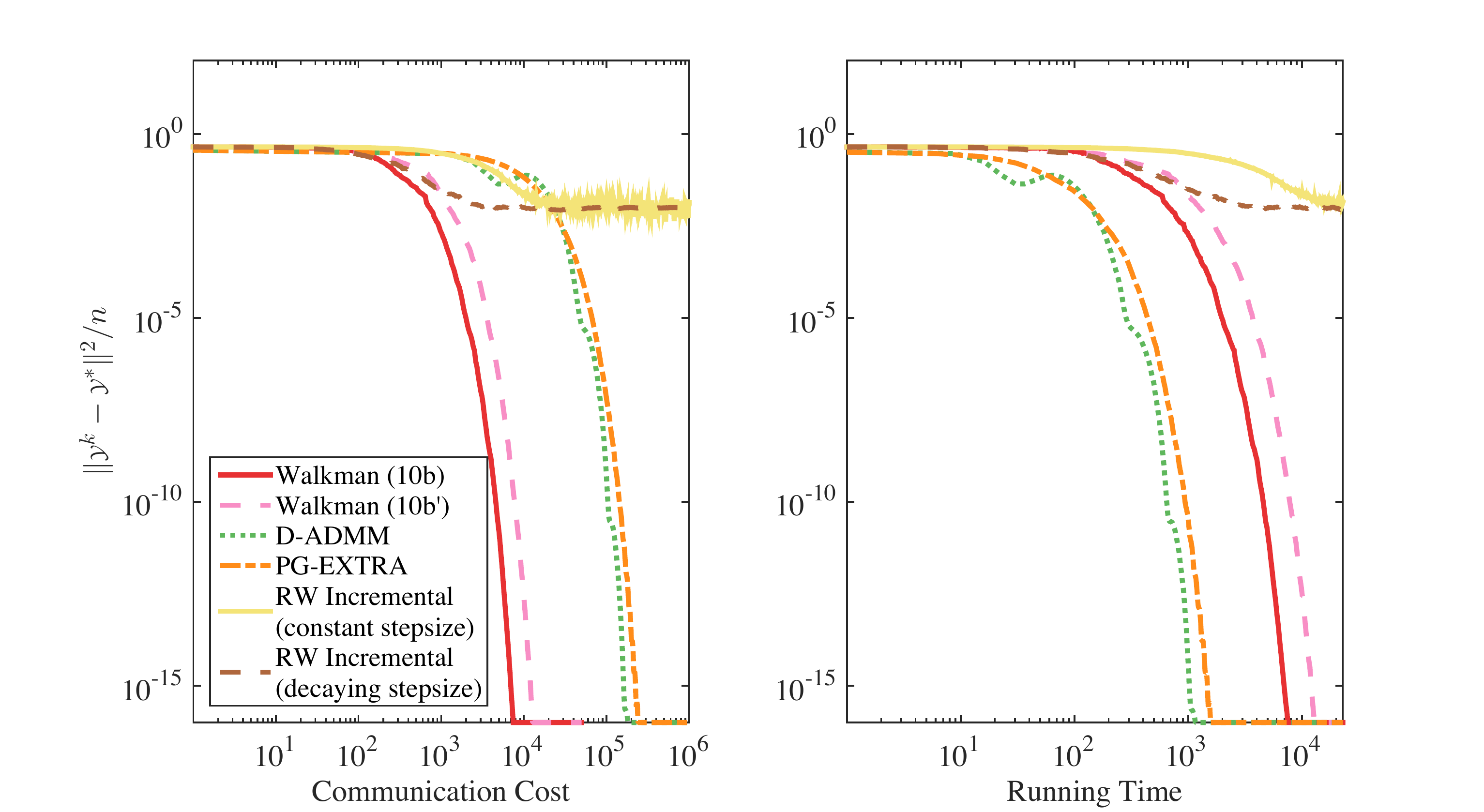}
        \caption{\footnotesize Performance of decentralized algorithms on logistic regression.}
        \label{fig:logistic-regression}
\end{figure}

{\color{black}\subsection{Decentralized non-negative principal component analysis}}
To test the performance on solving nonconvex, nonsmooth problem, the third experiment solves the Non-Negative Principal Component Analysis (NN-PCA) problem
\begin{align}\label{eqnn_pca}
\Min_{x\in\RR^p} ~&\frac{1}{n}\sum_{i=1}^n -x\tran \big(\frac{1}{b}\sum_{j=1}^b y_{ij} y_{ij}\tran\big) x ,\\
\St ~&\|x\|\le 1, x_i\ge 0, \forall i \in \{1,\cdots, p\}.\nonumber
\end{align}
where $y_{ij}\in\RR^p$ denotes the $j$-th sample kept by the $i$-th agent, and there are $b$ samples kept by each agent.
The objective function of \eqref{eqnn_pca}, named as $f$, forms the smooth part of standard optimization problem \eqref{eq:mainprob}, and ${\bf 1}_C$, the indicator function of the feasible space forms the nonsmooth part $r$.
In this experiment, we utilize the training set of the MNIST \cite{mnist} dataset to form the samples, and set $b=1000$.
Each agent only keeps samples with a same label.
{\color{black} Noticing that the NN-PCA problem is nonconvex, we use optimality gap to measure the distance between the algorithmic variables to problem's saddle points}, which is defined as
$$
\|\text{proj}_{\partial_{r}(x^k)}(-\nabla f(x^k))+\nabla f(x^k)\|^2 +\|sy^k-\mathds{1}\otimes x^k\|^2,
$$
where the first term measures how close is $\partial_r(x^k)+\nabla f(x^k)$ to $0$, and the second term measures the consensus violation of the copies kept by agents.
For PG-EXTRA and D-ADMM, since there is only $\sy^k$, we take $x^k$ as the mean of $\{y_1^k, \cdots, y_n^k\}.$
For RW Incremental methods, since there is only $x^k$, the second term of optimality gap is naturally $0$.
We run the simulation over the same network as the above two problems.
Under either  optimality criterion, the communication efficiency of Walkman is also observed in Fig. \ref{fig:nnpca}.

\begin{figure}[h]
        \centering
        \includegraphics[width= 0.5\textwidth]{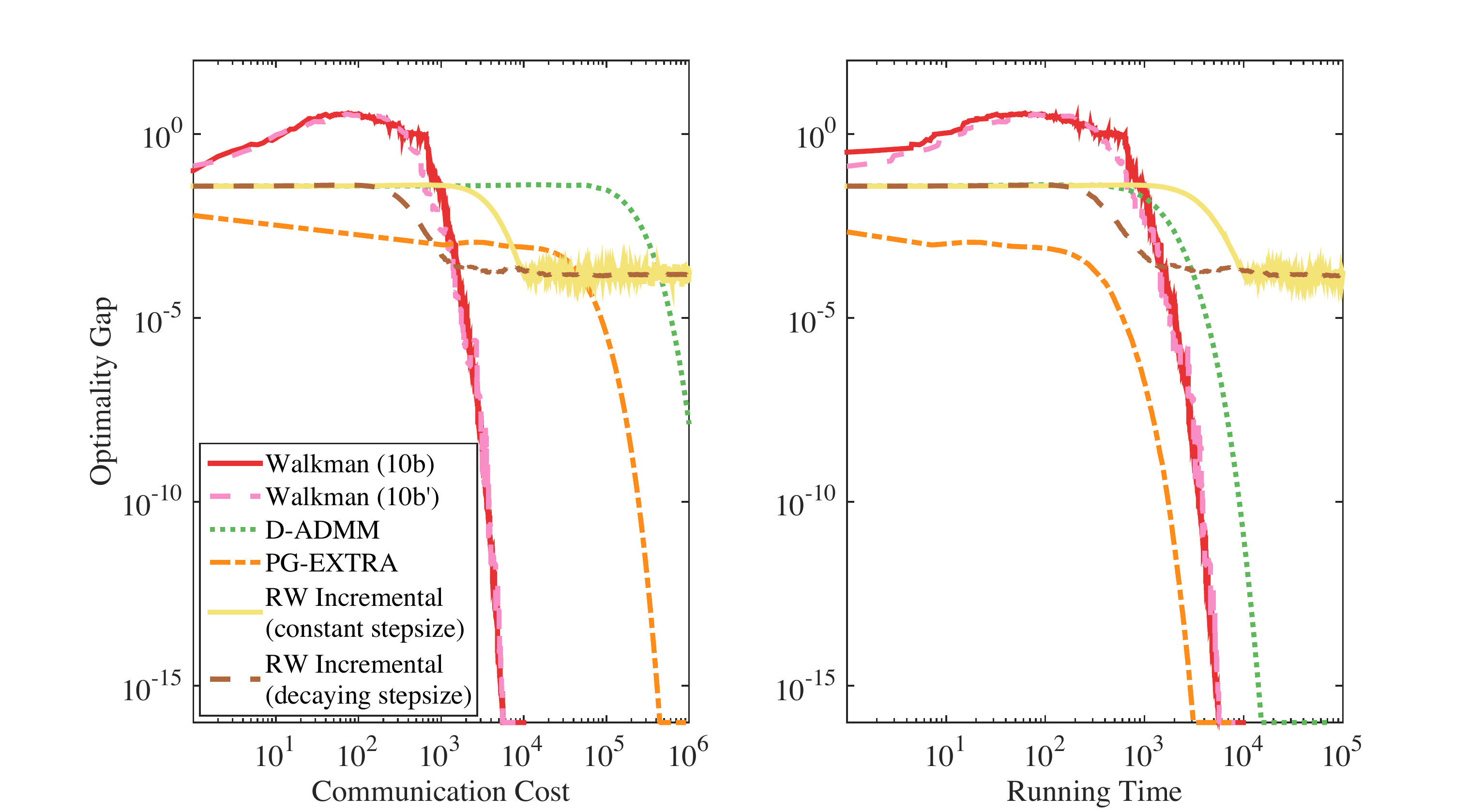}
        \caption{\footnotesize Performance of decentralized algorithms on NN-PCA.}
        \label{fig:nnpca}
\end{figure}

\section{Conclusion}\label{sc:con}
We have proposed a (random) walk algorithm, called Walkman, for decentralized consensus optimization.
The (random) walk carries the current solution $x$ and lets it updated by every visited agent.
Any limit point of the sequence of $x$ is almost surely a stationary point.
Under convexity assumption, the sequence converges to the optimal solution with a fixed step-size, which makes Walkman more efficient than the existing random-walk algorithms.
We have found Walkman uses less total communication than popular algorithms such as D-ADMM, EXTRA, exact diffusion, and PG-EXTRA though taking longer wall-clock time to converge. Random walks also add another layer of privacy protection.

\appendices
\section{Proof of Lemma \ref{lm:1}}
The proof of Lemma \ref{lm:1} takes a few steps, Lemmas \ref{lm:z_boundby_y}--\ref{lemma:yz_descent}.

Lemma \ref{lm:z_boundby_y} shows that the update on the dual variable can be bounded by that of the primal variable.
\begin{lemma}\label{lm:z_boundby_y}
        Under Assumption \ref{ass-lipschitz}, $(x^k, \sy^k, \sz^k)_{k> T}$, the sequence generated by Walkman iteration \eqref{eq:admm3'}, satisfies
        \begin{enumerate}
        \item if Walkman uses \eqref{eq:admm3y'}, it holds
      \begin{align}
                \| \sz^{k+1} - \sz^{k}\| =\|z_{i_k}^{k+1} \hspace{-0.8mm}- \hspace{-0.8mm}z_{i_k}^k\|\le L \|\sy^{k+1} - \sy^{k}\|; \label{eq-general-case}
         \end{align}
         \item if Walkman uses \eqref{eq:admmy-linearized}, it hold
          \begin{align}
                \| \sz^{k+1} - \sz^{k}\|=\|z_{i_k}^{k+1} \hspace{-0.8mm}- \hspace{-0.8mm}z_{i_k}^k\| \le L \|\sy^{\tau(k, i_k)+1} - \sy^{\tau(k, i_k)}\|. \label{eq-general-case-l}
         \end{align}
         \end{enumerate}
         \end{lemma}
\begin{proof}
Part 1) Remember   agent $i_k$ is activated at iteration $k$.  The optimality condition of \eqref{eq:admm3y'} for $i=i_k$ implies
        \begin{align}
                \nabla f_i(y_{i_k}^{k+1}) -\big(z_{i_k}^k+ \beta (x^{k+1} - y_{i_k}^{k+1})\big) = 0.
        \end{align}
        Substituting the above  into \eqref{eq:admm3z'} yields
        \begin{align}
                \nabla f_i(y_i^{k+1}) =  z_i^{k+1}, \quad \mbox{for }\ i = i_k. \label{eq:z_gradf}
        \end{align}
        Hence, for $i=i_k$, we have:
        \begin{align}
                \hspace{-1pt}\| z_i^{k+1} \hspace{-0.8mm}-\hspace{-0.8mm} z_i^{k}\| \hspace{-0.8mm} &\overset{(a)}{=}  \hspace{-0.8mm}\| z_i^{k+1} \hspace{-0.8mm}-\hspace{-0.8mm} z_i^{\tau(k, i)+1}\|  \hspace{-0.8mm}\overset{\eqref{eq:z_gradf}}{=}  \hspace{-0.8mm}\| \nabla f_i(y_i^{k+1})  \hspace{-0.8mm}-  \hspace{-0.8mm}\nabla f_i(y_i^{\tau(k, i)+1})\| \nonumber \\
                & \overset{\eqref{eq-ass-lip}}{\le} L \|y_i^{k+1} - y_i^{\tau(k, i)+1}\| \overset{(b)}{=} L \|y_i^{k+1} - y_i^{k}\|, \label{eq-case-1}
        \end{align}
        where $\tau(k, i)$ is defined in \eqref{eq:tau_def}.  Equality $(a)$ holds because $z_{i}^{k} = z_{i}^{\tau(k, i)+1}$ and  $(b)$ holds because $y_{i}^{k} = y_{i}^{\tau(k, i)+1}$. On the other hand, when $i\neq i_k$, agent $i$ is not activated at $k$, so $\| z_i^{k+1} - z_i^{k}\| = L \|y_i^{k+1} - y_i^{k}\| = 0$, and we have \eqref{eq-general-case}.
        
 Par 2) Substituting \eqref{eq:admmy-linearized} into \eqref{eq:admm3z'} yields
  \begin{align}
                \nabla f_i(y_i^{k}) =  z_i^{k+1}, \quad \mbox{for }\ i = i_k. \label{eq:z_gradf-l}
        \end{align}
        Comparing \eqref{eq:z_gradf} and \eqref{eq:z_gradf-l} and using $z_{i}^{k} = z_{i}^{\tau(k, i)+1}$ and $y_{i}^{k} = y_{i}^{\tau(k, i)+1}$, we get \eqref{eq-general-case-l} using a similar derivation for \eqref{eq-case-1}.
\end{proof}

Lemma \ref{lemma:x_descent} shows that the $x$-update in Walkman, i.e., \eqref{eq:admm3x'}, provides sufficient descent of the augmented Lagrangian.

\begin{lemma} \label{lemma:x_descent}
         Recall $L_\beta^k$ defined in \eqref{eq:defLk}. Under Assumption \ref{ass-prox}, for $\beta>\gamma, {k\ge 0},$ the Walkman iterates satisfy
                \begin{align}
                        L_\beta^k - L_\beta(x^{k+1}, \sy^k; \sz^k) \ge \frac{\beta - \gamma}{2}\|x^k - x^{k+1}\|^2. \label{eq:xdescent_lb}
                \end{align}
\end{lemma}
\begin{proof}
          We rewrite the augmented Lagrangian  in \eqref{eq-Lagrangian}  as
        \begin{align}
        \hspace{-0.8mm}L_{\beta}\left( x, \sy; \sz \right) \hspace{-0.8mm}=\hspace{-0.8mm}r(x)\hspace{-0.8mm}+\hspace{-0.8mm} \frac{1}{n}\hspace{-0.8mm} \left( \hspace{-1.3mm} F(\sy) \hspace{-0.8mm}+\hspace{-0.8mm}  \frac{\beta}{2}\Big\|\mathds{1}\otimes x \hspace{-0.8mm}-\hspace{-0.8mm} \sy \hspace{-0.8mm}+\hspace{-0.8mm} \frac{\sz}{\beta}\Big\|^2\hspace{-0.8mm}-\hspace{-0.8mm} \frac{\|\sz\|^2}{2\beta} \hspace{-0.5mm}\right)\hspace{-1mm}. \label{eq-Lagrangian-equiv}
        \end{align}
        Applying the cosine identity $\|b+c\|^2 - \|a+c\|^2 = \|b-a\|^2 + 2\langle a+c, b-a \rangle$, we have
        \begin{align}
        &\ L_\beta^k - L_\beta(x^{k+1},\sy^k;\sz^k)-r(x^k) +r(x^{k+1}) \nonumber\\
        =&\ \frac{\beta}{2n}\big\|\mathds{1}\otimes x^k - \sy^k + \frac{\sz^k}{\beta}\big\|^2 - \frac{\beta}{2n}\big\|\mathds{1}\otimes x^{k+1} - \sy^k + \frac{\sz^k}{\beta}\big\|^2 \nonumber \\
        =&\ \frac{\beta}{2n}\sum_{i=1}^{n} \Big( \|x^k - x^{k+1}\|^2 + 2\big\langle x^{k+1} - y_i^k + \frac{z_i^k}{\beta}, x^k - x^{k+1} \big\rangle \Big) \nonumber \\
        &\hspace{-4.5mm}{\ge}\ \frac{\beta}{2}\|x^k - x^{k+1}\|^2 -\langle d^k, x^k-x^{k+1}\rangle,\label{eq-res-2}
        \end{align}
        where $d^k$ is defined as
         \begin{align}
          d^k:=- \frac{\beta}{n}\sum_{i=1}^{n}(x^{k+1}-y^k_i+\frac{z^k_i}{\beta})\overset{(a)}{\in} \partial r(x^{k+1})\label{eq:dk_def},
         \end{align}
         where $(a)$ comes from the optimality condition of \eqref{eq:admm3x'} .
        Assumption  \ref{ass-prox} states
        $r(x^{k}) + \frac{\gamma}{2}\|x^k - x^{k+1}\|^2 \ge r(x^{k+1}) + \langle d^k, x^k - x^{k+1} \rangle, $
        substituting which into \eqref{eq-res-2} gives us \eqref{eq:xdescent_lb}.
\end{proof}

In Lemma \ref{lemma:yz_descent}, we derive the lower bound of descent in the augmented Lagrangian over the updates of $\sy$ and $\sz$.
\begin{lemma} \label{lemma:yz_descent} Recall $L_\beta^k$ defined in \eqref{eq:defLk}. Under Assumption \ref{ass-lipschitz}, for any $k>T$,
\begin{enumerate}
\item
if $\beta\geq 2L+2$,  Walkman using \eqref{eq:admm3y'} satisfies
        \begin{align}
                L_\beta(x^{k+1}, \sy^{k}; \sz^{k}) - L_\beta^{k+1} \geq\frac{1}{n}\|\sy^k - \sy^{k+1}\|^2.\label{eq:yz_update_descent}
        \end{align}
\item
if $\beta> L$,
Walkman using \eqref{eq:admmy-linearized} satisfies
        \begin{align}
   &L_\beta(x^{k+1}, \sy^{k}; \sz^{k}) - L_\beta^{k+1}\nonumber\\
&\geq\frac{\beta\hspace{-0.8mm} -\hspace{-0.8mm} L}{2n}\|\sy^k \hspace{-0.8mm} -\hspace{-0.8mm}  \sy^{k+1}\|^2\hspace{-0.8mm}-\hspace{-0.8mm}  \frac{L^2}{n\beta}\|\sy^{\tau(k, i_k)+1}\hspace{-0.8mm}  - \hspace{-0.8mm} \sy^{\tau(k, i_k)}\|^2.\label{eq:yz_update_descent-l}
        \end{align}

\end{enumerate}
\end{lemma}
\begin{proof}
From the Lagrangian \eqref{eq-Lagrangian}, we derive
        \begin{align}
        &\ L_\beta(x^{k+1}, \sy^{k}; \sz^{k}) - L_\beta^{k+1} \nonumber \\
        =&\ \frac{1}{n}\Big( f_{i_k}(y_{i_k}^k) + \langle z_{i_k}^{k}, x^{k+1} \hspace{-0.8mm}-\hspace{-0.8mm} y_{i_k}^{k}\rangle +\frac{\beta}{2}\|x^{k+1} - y^k_{i_k}\|^2 \nonumber \\
        &\hspace{-5pt}- f_{i_k}(y_{i_k}^{k+1}) -\langle z_{i_k}^{k+1}, x^{k+1} \hspace{-0.8mm}-\hspace{-0.8mm} y_{i_k}^{k+1}\rangle- \frac{\beta}{2} \|x^{k+1} - y_{i_k}^{k+1} \|^2\Big) \nonumber \\
        \overset{(a)}{=}&\ \frac{1}{n}\Big( f_{i_k}(y_{i_k}^k) - f_{i_k}(y_{i_k}^{k+1}) + \frac{\beta}{2}\|y_{i_k}^k - y_{i_k}^{k+1}\|^2 \nonumber\\
        &\quad\quad - \langle y_{i_k}^k - y_{i_k}^{k+1}, z_{i_k}^{k+1} \rangle - \frac{1}{\beta}\|z_{i_k}^{k+1} - z_{i_k}^k\|^2\Big) \label{eq:yz-common} \\
        \overset{(b)}{=}&\ \frac{1}{n}\Big( f_{i_k}(y_{i_k}^k) - f_{i_k}(y_{i_k}^{k+1}) + \frac{\beta}{2}\|y_{i_k}^k - y_{i_k}^{k+1}\|^2 \nonumber \\
        &\quad\quad - \langle y_{i_k}^k \hspace{-0.8mm}- \hspace{-0.8mm}y_{i_k}^{k+1}, \nabla f_{i_k}(y_{i_k}^{k+1}) \rangle \hspace{-0.8mm}- \hspace{-0.8mm}\frac{1}{\beta}\|z_{i_k}^{k+1} \hspace{-0.8mm}- \hspace{-0.8mm}z_{i_k}^k\|^2\Big) 
        \label{eq:z_substitute} \\
        \overset{(c)}{\ge}& \frac{1}{n}\Big(\hspace{-0.8mm}\hspace{-0.8mm}-\hspace{-0.8mm}\frac{L}{2}\|y_{i_k}^k \hspace{-0.8mm}-\hspace{-0.8mm} y_{i_k}^{k+1}\|^2 \hspace{-0.8mm}+\hspace{-0.8mm}\frac{\beta}{2}\|y_{i_k}^k \hspace{-0.8mm}-\hspace{-0.8mm} y_{i_k}^{k+1}\|^2 \hspace{-0.8mm}-\hspace{-0.8mm} \frac{1}{\beta}\|z_{i_k}^{k+1} - z_{i_k}^k\|^2\hspace{-0.8mm}\Big) \label{eq:yz-common1} \\
        \overset{(d)}{\ge}& \frac{1}{n}\Big(\hspace{-0.8mm}\hspace{-0.8mm}-\hspace{-0.8mm}\frac{L}{2}\|y_{i_k}^k \hspace{-0.8mm}-\hspace{-0.8mm} y_{i_k}^{k+1}\|^2 \hspace{-0.8mm}+\hspace{-0.8mm}\frac{\beta}{2}\|y_{i_k}^k \hspace{-0.8mm}-\hspace{-0.8mm} y_{i_k}^{k+1}\|^2 \hspace{-0.8mm}-\hspace{-0.8mm} \frac{L^2}{\beta}\|y_{i_k}^k - y_{i_k}^{k+1}\|^2\hspace{-0.8mm}\Big) \nonumber \\
        \overset{(e)}{\ge}&\ \frac{1}{n}\|y_{i_k}^k - y_{i_k}^{k+1}\|^2 = \frac{1}{n}\|\sy^k - \sy^{k+1}\|^2.
        \end{align}
        where equality (a) holds due to$\|b+c\|^2 - \|a+c\|^2 = \|b-a\|^2 + 2\langle a+c, b-a \rangle$ and recursion \eqref{eq:admm3z'}, equality (b) holds because of \eqref{eq:z_gradf}, inequality (c) holds because  $f_i(\cdot)$ is $L$-Lipschitz differentiable, inequality (d) holds because of  \eqref{eq-case-1}, and inequality (e) follows from the assumption $\beta \ge 2L+2$. 
        
        Next, we study Walkman using \eqref{eq:admmy-linearized}. The above equation array holds to \eqref{eq:yz-common}.
        By substituting \eqref{eq:z_gradf-l} into \eqref{eq:yz-common}, we get
        \begin{align}
        &\ L_\beta(x^{k+1}, \sy^{k}; \sz^{k}) - L_\beta^{k+1} \nonumber \\
  {=}&\ \frac{1}{n}\Big( f_{i_k}(y_{i_k}^k) - f_{i_k}(y_{i_k}^{k+1}) + \frac{\beta}{2}\|y_{i_k}^k - y_{i_k}^{k+1}\|^2 \nonumber \\
        &\quad\quad - \langle y_{i_k}^k \hspace{-0.8mm}- \hspace{-0.8mm}y_{i_k}^{k+1}, \nabla f_{i_k}(y_{i_k}^{k}) \rangle \hspace{-0.8mm}- \hspace{-0.8mm}\frac{1}{\beta}\|z_{i_k}^{k+1} \hspace{-0.8mm}- \hspace{-0.8mm}z_{i_k}^k\|^2\Big) \label{eq:z_substitute-l}.
        \end{align}
       While \eqref{eq:z_substitute} has $\nabla f_{i_k}(y_{i_k}^{k+1})$,  \eqref{eq:z_substitute-l} involves $\nabla f_{i_k}(y_{i_k}^{k})$. However, from \eqref{eq:z_substitute-l}, using  $\nabla f_{i_k}(\cdot)$ being $L$-Lipschitz, we still get \eqref{eq:yz-common1}, to which we can apply Lemma \ref{lm:z_boundby_y} 2) to get \eqref{eq:yz_update_descent-l}.
\end{proof}

In Lemma \ref{lm:lyapunov_descent}, we establish the sufficient descent in Lyapunov functions of Walkman.
\begin{lemma}\label{lm:lyapunov_descent}
Recall $L_\beta^k$ and $M_{\beta}^k$ defined in \eqref{eq:defLk} and \eqref{eq:defMk}. Under Assumptions \ref{ass-lipschitz} and \ref{ass-prox}, for  any $k>T$,
\begin{enumerate}
\item if $\beta> \max\{\gamma, 2L+2\}$, Walkman using \eqref{eq:admm3y'} satisfies
\begin{align}
L_{\beta}^k\hspace{-0.8mm} - \hspace{-0.8mm}L_{\beta}^{k+1}\hspace{-0.8mm}\ge\hspace{-0.8mm} \frac{\beta-\gamma}{2}\|x^k \hspace{-0.8mm}-\hspace{-0.8mm} x^{k+1}\|^2 \hspace{-0.8mm}+ \hspace{-0.8mm}\frac{1}{n}\|\sy^k\hspace{-0.8mm} -\hspace{-0.8mm} \sy^{k+1}\|^2;\label{eq:Lk_descent}
\end{align}
\item if $\beta> \max\{\gamma, 2L^2+L+2\}$ the Walkman using \eqref{eq:admmy-linearized} satisfies
\begin{align}
\!\!\!\!M_{\beta}^k \hspace{-0.8mm}-\hspace{-0.8mm} M_{\beta}^{k+1} \hspace{-0.8mm}\ge& \frac{\beta-\gamma}{2}\|x^k \hspace{-0.8mm}-\hspace{-0.8mm} x^{k+1}\|^2 \hspace{-0.8mm}+ \hspace{-0.8mm}\frac{1}{n}\|\sy^k\hspace{-0.8mm} -\hspace{-0.8mm} \sy^{k+1}\|^2\notag\\
&\hspace{-0.8mm} + \frac{L^2}{2n}\|\sy^{\tau (k,i_k)+1}\hspace{-0.8mm} -\hspace{-0.8mm} \sy^{\tau (k,i_k)}\|^2.\label{eq:Mk_descent}
\end{align}
\end{enumerate}
\end{lemma}
\begin{proof}
Statement 1) is a direct result of adding \eqref{eq:xdescent_lb} and \eqref{eq:yz_update_descent}. 
To prove statement 2),  noticing
\begin{equation}
y_{i}^{\tau (k+1,i)+1}\hspace{-0.8mm} - \hspace{-0.3mm}y_{i}^{\tau (k+1,i)} \hspace{-0.3mm}=\hspace{-0.3mm}\begin{cases} y_i^{k+1}\hspace{-0.8mm} - \hspace{-0.8mm}y_i^k, &\hspace{-3mm} i=i_k\\ y_{i}^{\tau (k,i)+1}\hspace{-0.8mm} -\hspace{-0.8mm} y_{i}^{\tau (k,i)},&\hspace{-3mm}\text{otherwise},\end{cases}
\end{equation} 
we derive
\begin{align}
&M_{\beta}^k - M_{\beta}^{k+1} \notag\\
= &L_{\beta}^k - L_{\beta}^{k+1} +\frac{L^2}{n}\big(\|y_{i_k}^{\tau (k,i_k)+1}\hspace{-0.8mm} -\hspace{-0.8mm} y_{i_k}^{\tau (k,i_k)}\|^2-\|y_{i_k}^{k+1}\hspace{-0.8mm} - \hspace{-0.8mm}y_{i_k}^k\|^2\big)\notag\\
= &L_{\beta}^k - L_{\beta}^{k+1} +\frac{L^2}{n}\big(\|\sy^{\tau (k,i_k)+1}\hspace{-0.8mm} -\hspace{-0.8mm} \sy^{\tau (k,i_k)}\|^2-\|\sy^{k+1}\hspace{-0.8mm} - \hspace{-0.8mm}\sy^k\|^2\big).\label{eq:mk_descent_mid}
\end{align}
Substituting \eqref{eq:xdescent_lb} and \eqref{eq:yz_update_descent-l} into \eqref{eq:mk_descent_mid} and using $\frac{\beta}{2}-\frac{L}{2}-{L^2}\ge1$ and $1-\frac{1}{\beta}> \frac{1}{2}$, we complete the proof of statement 2).
\end{proof}

Lemma \ref{lm:lyapunov_func_lowerbounded} states that both Lyapunov functions are  lower bounded. 
\begin{lemma}\label{lm:lyapunov_func_lowerbounded}
For $\beta> \max\{\gamma, 2L+2\}$ (resp. $\beta> \max\{\gamma, 2L^2+L+2\}$),  Walkman using \eqref{eq:admm3y'} (resp. \eqref{eq:admmy-linearized}) ensures a lower bounded sequence $(L_{\beta}^k)_{k\ge 0}$ (resp. $(M_{\beta}^k)_{k\ge 0}$).
\end{lemma}
\begin{proof}
For Walkman using \eqref{eq:admm3y'} and  $k>T$, we have
\begin{align}
        L_{\beta}^k        &= r(x^k)+ \frac{1}{n}\sum_{j=1}^n\big(f_j(y_j^{k}) + \langle z_j^k, x^k - y_j^{k}\rangle\big)  \nonumber\\
        &~~~~~+\hspace{-0.8mm} \frac{\beta}{2n}\|\mathds{1}\otimes x^k \hspace{-0.8mm}-\hspace{-0.8mm} \sy^k\|^2 \nonumber \\
        &\overset{\eqref{eq:z_gradf}}{=}r(x^k)+ \frac{1}{n}\sum_{j=1}^n\big(f_j(y_j^{k}) + \langle \nabla f_{j}(y_{j}^{k}), x^k - y_j^{k}\rangle\big)  \nonumber\\
        &\hspace{5mm}+\hspace{-0.8mm} \frac{\beta}{2n}\|\mathbf{1}\hspace{-0.8mm}\otimes\hspace{-0.8mm} x^k\hspace{-0.8mm}-\hspace{-0.8mm} \sy^k\hspace{-0.4mm}\|^2  \nonumber\\
        &\overset{\rm (a)}{\geq} \hspace{-0.8mm}r(x^k)\hspace{-0.8mm}+ \hspace{-0.8mm} \frac{1}{n}\sum_{j=1}^n  f_j(x^k)  + \frac{\beta -L}{2n}\|\mathbf{1}\otimes x^k- \sy^k\|^2 \label{eq:lowerbound_L-1}\\
        & \geq \min_x \Big\{r(x)+ \frac{1}{n}\sum_{j=1}^n  f_j(x)\Big\} + \frac{\beta\hspace{-0.8mm} -\hspace{-0.8mm}L}{2n}\|\mathbf{1}\otimes x^k\hspace{-0.8mm}- \hspace{-0.8mm}\sy^k\|^2 \nonumber\\
&  \overset{\rm (b)}{>}\hspace{-0.8mm} -\infty, \nonumber
        \end{align}
        where (a) holds as each $f_j$ is Lipschitz differentiable and (b) from Assumption \ref{coercivity} and $\beta >L$. So, $L_{\beta}^k$ is lower bounded. 

Next, for Walkman using \eqref{eq:admmy-linearized} and $k>T$, we derive
\begin{align}
 &M_{\beta}^k\notag\\
\overset{\eqref{eq:z_gradf-l}}{=} & r(x^k)+ \frac{1}{n}\sum_{j=1}^n\big(f_j(y_j^{k}) + \langle  \nabla f_{j}(y_{j}^{\tau(k, j)}), x^k - y_j^{k}\rangle\big)  \nonumber\\
        &+\hspace{-0.8mm} \frac{\beta}{2n}\|\mathds{1}\otimes x^k \hspace{-0.8mm}-\hspace{-0.8mm} \sy^k\|^2+\frac{L^2}{n}\sum_{i=1}^{n}\| y_{i}^{\tau (k,i)+1} - y_{i}^{\tau (k,i)}\|^2 \nonumber \\
   \overset{(a)}{\ge}&r(x^k)\hspace{-0.8mm}+\hspace{-0.8mm}  \frac{1}{n}\sum_{j=1}^n\big(f_j(x^{k}) \hspace{-0.8mm}+\hspace{-0.8mm}  \langle  \nabla f_{j}(y_{j}^{\tau(k, j)})\hspace{-0.8mm}-\hspace{-0.8mm} \nabla f_{j}(y_{j}^{k}), x^k \hspace{-0.8mm}-\hspace{-0.8mm}  y_j^{k}\rangle\big)  \nonumber\\
        &+\hspace{-0.8mm} \frac{\beta-L}{2n}\|\mathds{1}\otimes x^k \hspace{-0.8mm}-\hspace{-0.8mm} \sy^k\|^2+\frac{L^2}{n}\sum_{i=1}^{n}\| y_{i}^{k} - y_{i}^{\tau (k,i)}\|^2 \nonumber \\
    \overset{(b)}{\ge} & r(x^k)+ \frac{1}{n}\sum_{j=1}^n\big(f_j(x^{k}) -\|\nabla f_{j}(y_{j}^{\tau(k, j)})- \nabla f_{j}(y_{j}^{k})\|^2\big)  \nonumber\\
        &+\hspace{-0.8mm} \frac{\beta-L-2}{2n}\|\mathds{1}\otimes x^k \hspace{-0.8mm}-\hspace{-0.8mm} \sy^k\|^2+\frac{L^2}{n}\sum_{i=1}^{n}\| y_{i}^{k} - y_{i}^{\tau (k,i)}\|^2 \nonumber \\
          \overset{(c)}{\ge} & \min_{x}\Big\{r(x^k)+ \frac{1}{n}\sum_{j=1}^n f_j(x^{k})\Big\} +\hspace{-0.8mm} \frac{2L^2}{2n}\|\mathds{1}\otimes x^k \hspace{-0.8mm}-\hspace{-0.8mm} \sy^k\|^2\label{eq:Mk_lb}
\\
          \overset{(d)}{>}&-\infty,\notag
          \end{align}
where (a) holds because each $f_j$ is Lipschitz differentiable, (b) holds due Young's inequality, (c) follows from the assumption $\beta>2L^2+L+2$ and the Lipschitz smoothness of each $f_j$, and (d) holds due to Assumption \ref{coercivity}. Therefore, $M_{\beta}^k$ is bounded from below.
\end{proof}

With  above lemmas, we are ready to prove Lemma \ref{lm:1}.
\begin{proof}[Proof of Lemma \ref{lm:1}]
Recall that the maximal hitting time $T$ is almost surely finite.
The monotonicity of $(L_{\beta}^k)_{k> T}$ (resp. $(M_{\beta}^k)_{k>T}$) in Lemma \ref{lm:lyapunov_descent} and their lower boundedness in Lemma \ref{lm:lyapunov_func_lowerbounded} ensure convergence of $(L_\beta^k)_{k\ge 0}$ (resp. $(M_\beta^k)_{k\ge 0}$).
 
For statement 2), 
We first consider Walkman with \eqref{eq:admm3y'}.
By statement 1) and \eqref{eq:lowerbound_L-1}, $r(x^k)+\frac{1}{n}F(\vx^k)$  is upper bounded by $\max\{\max_{t\in\{0,\cdots, T\}}\{r(x^t)+\frac{1}{n}F(\vx^t)\}, L_{\beta}^{T+1}\}$, and $\|\mathbf{1}\otimes x^k- \sy^k\|^2$ is upper bounded by $\max\{\max_{t\in\{0,\cdots, T\}}\{\|\mathbf{1}\otimes x^t- \sy^t\|^2\}, L_{\beta}^{T+1}\}$.
By Assumption \ref{coercivity},
the  sequence $(x^k)$ is bounded. The boundedness of $\|\mathbf{1}\otimes x^k- \sy^k\|^2$ further leads to that of $(\sy^k)$. Finally,  \eqref{eq:z_gradf} and Assumption \ref{ass-lipschitz} ensure $(\sz^k)$ is bounded, too. Altogether, $(x^k,\sy^k,\sz^k)$ is bounded.
Starting from statement 1)' and \eqref{eq:Mk_lb}, a similar argument leads to boundedness of $(x^k,\sy^k,\sz^k)$  for Walkman using \eqref{eq:admmy-linearized}.
\end{proof}

\section{Proof of Lemma \ref{lm:gk_convg}}
Following the aforementioned proof idea, we provide the detailed proof of Lemma \ref{lm:gk_convg} in this Section.
\begin{proof}[Proof of Lemma \ref{lm:gk_convg}]
First, 
recall  Lemma \ref{lm:lyapunov_descent} and $T<\infty$, for Walkman using \eqref{eq:admm3y'}, we have
        \begin{align}\label{eq:vardiff_lim}
        \sum_{k=0}^{\infty} \left(\EE\|x^k - x^{k+1}\|^2+ \EE \|\sy^{k} - \sy^{k+1}\|^2\right)< +\infty;
        \end{align}
 and for Walkman using \eqref{eq:admmy-linearized},
 \begin{align}\label{eq:vardiff_lim-l}
        \sum_{k=0}^{\infty} \Big(&\EE\|x^k - x^{k+1}\|^2+ \EE \|\sy^{k} - \sy^{k+1}\|^2 \notag\\
        &+\EE \|{\sy^{\tau(k, i_k)} - \sy^{\tau(k, i_k)+1}}\|_2^2\Big)< +\infty,
        \end{align}
Hence, by Lemma \ref{lm:lyapunov_descent},
\begin{align}
        \sum_{k=0}^{\infty} \Big(\EE\|x^k - x^{k+1}\|^2+& \EE \|\sy^{k} - \sy^{k+1}\|^2 \notag\\
        &+\EE \|\sz^{k} - \sz^{k+1}\|_2^2\Big)< +\infty,\label{eq:summable}
\end{align}
holds for Walkman using either \eqref{eq:admm3y'} or \eqref{eq:admmy-linearized}.

The proof starts with computing the subdifferentials of the augmented Lagrangian \eqref{eq-Lagrangian-equiv} with the updates in \eqref{eq:admm3'}:
\begin{align}
\partial_x L_{\beta}^{k+1}~{\ni} ~& d^k-\frac{\beta}{n}(y_{i_k}^{k+1} - y_{i_k}^k) +\frac{1}{n}(z_{i_k}^{k+1} - z_{i_k}^k),\nonumber\\
\overset{\eqref{eq:dk_def}}{=}&\hspace{-0.8mm}-\frac{\beta}{n}(y_{i_k}^{k+1} - y_{i_k}^k) +\frac{1}{n}(z_{i_k}^{k+1} - z_{i_k}^k)=: w^k,\label{eq:partialx_def}\\
\nabla_{y_j}  L_\beta^{k+1} = &\frac{1}{n} \left(\nabla f_{j}(y_{j}^{k+1}) - z_{j}^{k+1} + \beta (y_{j}^{k+1} - x^{k+1})\right),\label{eq:partialy_def}\\
\nabla_{z_j}  L_\beta^{k+1} =& \frac{1}{n}\left(x^{k+1}-y_{j}^{k+1}\right).\label{eq:partialz_def}
\end{align}
For notational brevity, we define $g^k$ and $q_{\color{black}i}^k$ as
\begin{align}\label{eq:gk_def}
                g^k  := \left[\begin{array}{c} w^k \\ \nabla_{\ssy} L_\beta^{k+1}\\ \nabla_{\ssz} L_\beta^{k+1} \end{array}\right], ~
                q_{i}^{k}:=\left[\begin{array}{c} w^k\\ \nabla_{y_{i}} L_\beta^{k+1} \\ \nabla_{z_i} L_\beta^{k+1}\end{array}\right],
\end{align}
        where $i\in V$ is an agent index, and $g^k$ is the gradient of $L_\beta^{k+1}$.
For  $\delta\in(0,1)$ and $k\geq \uptau(\delta)+1$, by the triangle inequality:
\begin{align}
\|q_{i_k}^{k-\uptau(\delta)-1}\|^2=& \|q_{i_k}^{k-\uptau(\delta)-1} - q_{i_k}^{k}+q_{i_k}^{k}\|^2\nonumber\\
\leq & 2\underbrace{\|q_{i_k}^{k-\uptau(\delta)-1} - q_{i_k}^{k}\|^2}_{A} + 2\underbrace{\|q_{i_k}^{k}\|^2}_{B}\label{eq:q_bound1}
\end{align}
Below, we upper bound $A$ and $B$ separately. 
%
$A$ has three parts corresponding to the three components of $g$. Its first part is
\begin{align}\label{eq:wk_diff}
\|w^{k-\uptau(\delta)-1}  &- w^{k}\|^2 \leq 2\|w^{k-\uptau(\delta)-1} \|^2 + 2\| w^k\|^2 \nonumber \\
\leq \hspace{-0.8mm} \frac{4}{n^2}\Big(\beta^2\|\sy^{k+1}  \hspace{-0.8mm}&- \sy^k\|^2 \hspace{-0.8mm}+ \hspace{-0.8mm}\beta^2\|\sy^{k-\uptau(\delta)} \hspace{-0.8mm} -  \hspace{-0.8mm}\sy^{k-\uptau(\delta)-1}\|^2,\notag\\
+\|\sz^{k+1}  \hspace{-0.8mm}&- \sz^k\|^2 \hspace{-0.8mm}+ \hspace{-0.8mm}\|\sz^{k-\uptau(\delta)} \hspace{-0.8mm} -  \hspace{-0.8mm}\sz^{k-\uptau(\delta)-1}\|^2\Big)
\end{align}
where the 2nd inequality follows from \eqref{eq:partialx_def}.
Then by \eqref{eq:partialy_def}, we bound the 2nd part of $A$
\begin{align}
&\|\nabla_{y_{i_k}} L_{\beta}^{k-\uptau(\delta)-1} -\nabla_{y_{i_k}} L_{\beta}^{k+1}\|^2 \nonumber\\
\overset{\rm(a)}{\leq}& \frac{4L^2 + 4\beta^2}{n^2}\|y_{i_k}^{k-\uptau(\delta)-1}-y_{i_k}^{k+1}\|^2 + \frac{4}{n^2}\|z_{i_k}^{k-\uptau(\delta)-1} - z_{i_k}^{k+1}\|^2 \nonumber\\
&+ \frac{4\beta^2}{n^2}\|x^{k-\uptau(\delta)-1}-x^{k+1}\|^2\nonumber \\
\leq & \frac{(\uptau(\delta)+2)(4+4\beta^2+4L^2)}{n^2}\hspace{-3mm}\sum_{t=k-\uptau(\delta)-1}^{k}\hspace{-3mm}\Big(\|x^t -x^{t+1}\|^2\nonumber\\
&+\|\sy^t -\sy^{t+1}\|^2 +\|\sz^t -\sz^{t+1}\|^2\Big),\label{eq:partialy_diff}
\end{align}
where (a) uses the inequality of arithmetic and geometric means and the Lipschitz differentiability of $f_j$  in Assumption \ref{ass-lipschitz}.
From \eqref{eq:partialz_def}, the 3rd part of $A$ can be bounded as
\begin{align}
&\|\nabla_{z_{i_k}} L_{\beta}^{k-\uptau(\delta)-1} -\nabla_{z_{i_k}} L_{\beta}^{k+1}\|^2\nonumber \\
\leq &\frac{2}{n^2}(\|x^{k-\uptau(\delta)-1}-x^{k+1}\|^2 + \|y_{i_k}^{k-\uptau(\delta)-1}-y_{i_k}^{k+1}\|^2)\nonumber\\
\leq & \frac{2(\uptau(\delta)+2)}{n^2}\hspace{-4.5mm}\sum_{t=k-\uptau(\delta)-1}^{k}\hspace{-4.5mm}\Big(\|x^t -x^{t+1}\|^2+\|\sy^t -\sy^{t+1}\|^2\Big).\label{eq:partialz_diff}
\end{align}
Substituting  \eqref{eq:wk_diff}, \eqref{eq:partialy_diff} and \eqref{eq:partialz_diff} into term $A$, we get a constant $C_1$ $\sim O(\frac{\uptau(\delta)+1}{n^2})$, depending on $\uptau(\delta), \beta, L$ and $n$, such that
\begin{align}\label{eq:partA_bound}
A \hspace{-0.8mm}\leq  \hspace{-0.8mm}C_1\hspace{-7mm}\sum_{t=k-\uptau(\delta)-1}^{k}\hspace{-7mm}\big(\|x^t \hspace{-0.8mm}-\hspace{-0.8mm}x^{t+1}\|^2\hspace{-0.8mm}+\hspace{-0.8mm}\|\sy^t\hspace{-0.8mm} -\hspace{-0.8mm}\sy^{t+1}\|^2 \hspace{-0.8mm}+\hspace{-0.8mm} \|\sz^t\hspace{-0.8mm} -\hspace{-0.8mm}\sz^{t+1}\|^2\big).
\end{align}
To bound the term $B$, using \eqref{eq:partialz_def} and \eqref{eq:admm3z'}, we have
\begin{align}
\nabla_{z_{i_k}} L_\beta^{k+1} {=}& \frac{1}{n\beta}(z_{i_k}^{k+1} - z_{i_k}^k), \label{eq:partialzik_def}
\end{align}
Applying \eqref{eq:partialy_def}, \eqref{eq:z_gradf} and \eqref{eq:z_gradf-l}, we derive $\nabla_{y_{i_k}}  L_\beta^{k+1}$ for Walkman using \eqref{eq:admm3y'} or \eqref{eq:admmy-linearized}:
\begin{align}
\text{\eqref{eq:admm3y'}:}~\nabla_{y_{i_k}}  L_\beta^{k+1} {=} &\frac{1}{n}\left(z_{i_k}^k - z_{i_k}^{k+1}\right),\label{eq:partialyik_def}\\
\text{\eqref{eq:admmy-linearized}:}~\nabla_{y_{i_k}} \hspace{-2pt} L_\beta^{k+1} {=} &\frac{1}{n}\big( \nabla f_{i_k}(y_{i_k}^{k+1}) -\nabla f_{i_k}(y_{i_k}^k)+z_{i_k}^k - z_{i_k}^{k+1} \big).
\end{align}
For both, we have
\begin{align}\label{eq:partB_bound}
B \leq C_2\left(\|\sy^{k+1} - \sy^k\|^2+\|\sz^{k+1} - \sz^k\|^2\right),
\end{align}
for a constant $C_2$ depending on $L, \beta$ and $n$, { in the order of {$\sim O(\frac{1}{n^2})$}}.
Then substituting \eqref{eq:partA_bound} and \eqref{eq:partB_bound} into \eqref{eq:q_bound1} and taking expectations yield
\begin{align}
\EE\|q_{i_k}^{k-\uptau(\delta)-1}\|^2\hspace{-0.8mm}\leq\hspace{-0.8mm} C\hspace{-6mm}\sum_{t=k-\uptau(\delta)-1}^{k}\hspace{-5mm}\Big(&\EE\|x^t\hspace{-0.8mm}-\hspace{-0.8mm}x^{t+1}\|^2\hspace{-0.8mm}+\hspace{-0.8mm}\EE\|\sy^t \hspace{-0.8mm}-\hspace{-0.8mm}\sy^{t+1}\|^2\notag\\
&+\EE\|\sz^t -\sz^{t+1}\|^2\Big),\label{eq:q_bound_ori}
\end{align}
where $C=C_1\hspace{-0.8mm}+\hspace{-0.8mm}C_2$, { and one has $C\sim O(\frac{\uptau(\delta)+1}{n^2})$}. 
Recalling \eqref{eq:summable}, we get the convergence
\begin{align}\label{eq:qlim}
\lim_{k\to\infty}\EE\|q_{i_k}^{k-\uptau(\delta)-1}\|^2=0,
\end{align}
which completes the proof of step 1).

In step 2), we compute the conditional expectation:
        \begin{align}
        &\EE\left(\|q_{i_k}^{k-\uptau(\delta)-1}\|^2\mid\chi^{k-\uptau(\delta)}\right) \nonumber\\
        = &\sum_{j=1}^n [\vP^{\uptau(\delta)}]_{i_{k-\uptau(\delta)}, j}\Big(\|\nabla_x L_{\beta}^{k-\uptau(\delta)} \|^2 + \|\nabla_{y_j}L_\beta^{k-\uptau(\delta)} \|^2 \nonumber\\
        &+\|\nabla_{z_j}L_\beta^{k-\uptau(\delta)} \|^2\Big)\nonumber\\
        \overset{\rm (a)}{\geq}&(1-\delta)\pi_*\|g^{k-\uptau(\delta)-1}\|^2,\label{eq:q_bound_g}
        \end{align}
where (a) follows from \eqref{eq:mix_time_prop} and the definition of $g^k$  in \eqref{eq:gk_def}.
Then, with \eqref{eq:qlim}, it holds
        \begin{align}
        \lim_{k\to\infty}\EE\|g^{k}\|^2=\lim_{k\to\infty}\EE\|g^{k-\uptau(\delta)-1}\|^2 =0.
        \end{align}
        By the Schwarz inequality $\big(\EE\|g^{k}\|\big)^2\leq \EE\|g^{k}\|^2$, we have
        \begin{align}\label{eg->0}
                \lim_{k\to \infty} \EE\|g^k\| = 0.
        \end{align}
 Next, we prove step 3). By Markov's inequality, for each $\epsilon>0$, it holds that
        \begin{align}
        \Pr(\|g^{k}\|\geq\epsilon)\leq \frac{ \EE\|g^{k}\|}{\epsilon} \quad\overset{\eqref{eg->0}}{\Rightarrow}\quad \lim_{k\to\infty} \Pr(\|g^{k}\|\geq\epsilon) = 0.\label{eq:markov_ineq}
        \end{align}
When a subsequence $(k_s)_{s\ge 0}$ is provided, \eqref{eq:markov_ineq} implies.
 \begin{align}
 \lim_{s\to\infty} \Pr(\|g^{k_s}\|\geq\epsilon) = 0
 \end{align}
Then, for $j\in\NN$, select $\epsilon = 2^{-j}$ and we can find a nondecreasing subsubsequence $(k_{s_j})$, such that
\begin{align}
\Pr(\|g^{k_{s}}\|\ge 2^{-j})\le 2^{-j},\quad \forall k_s\ge k_{s_j}.
\end{align}
Since,
\begin{align}
\sum_{j=1}^\infty \Pr(\|g^{k_{s_j}}\|\ge 2^{-j})\le \sum_{j=1}^\infty 2^{-j} = 1,\label{eq:pre_bc}
\end{align}
the Borel-Cantelli lemma yields
\begin{align}
\Pr\big(\limsup_j\{\|g^{k_{s_j}}\|\ge 2^{-j}\} \big)=0,
\end{align}
and thus
\begin{align}
\Pr\big(\lim_j\|g^{k_{s_j}}\|=0\big) = 1. \label{eq:lm2_res}
\end{align}
This completes step 3) and thus the entire Lemma \ref{lm:gk_convg}.
\end{proof}

{\color{black} 
\section{Proof of Theorem \ref{thm:convg_rate_nonconvex}}\label{appendixC}
We provide the detailed proof of Theorem \ref{thm:convg_rate_nonconvex}  in this section.
\begin{proof}[Proof of Theorem \ref{thm:convg_rate_nonconvex}]
It can be simply verified that under the specific initialization, \eqref{eq:z_gradf} and \eqref{eq:z_gradf-l} hold for all $k\ge 0$, and consequently ensure Lemmas \ref{lm:z_boundby_y}-\ref{lm:lyapunov_descent} hold for all $k\ge 0$. 
For $g^k$ defined in \eqref{eq:gk_def}, \eqref{eq:q_bound_ori} and \eqref{eq:q_bound_g} hold.
Jointly applying \eqref{eq:q_bound_ori} and \eqref{eq:q_bound_g}, for any $k\ge \uptau(\delta)+1$, one has
\begin{align}
\EE \|g^{k-\uptau(\delta)-1}\|^2\le &\frac{C}{(1-\delta)\pi_*}\hspace{-1.5mm}\sum_{t=k-\uptau(\delta)-1}^{k}\hspace{-5mm}\Big(\EE\|x^t\hspace{-0.8mm}-\hspace{-0.8mm}x^{t+1}\|^2\hspace{-0.8mm}\notag\\
&+\hspace{-0.8mm}\EE\|\sy^t \hspace{-0.8mm}-\hspace{-0.8mm}\sy^{t+1}\|^2\hspace{-0.8mm}+\hspace{-0.8mm}\EE\|\sz^t -\sz^{t+1}\|^2\Big)
\end{align}
According to Lemmas \ref{lm:z_boundby_y} and \ref{lm:lyapunov_descent}, for Walkman using \eqref{eq:admm3y'} and $k\ge \uptau(\delta)+1$, it holds,
\begin{align}
&\sum_{t=k-\uptau(\delta)-1}^{k}\hspace{-5mm}\Big(\EE\|x^t\hspace{-0.8mm}-\hspace{-0.8mm}x^{t+1}\|^2\hspace{-0.8mm}+\hspace{-0.8mm}\EE\|\sy^t \hspace{-0.8mm}-\hspace{-0.8mm}\sy^{t+1}\|^2\hspace{-0.8mm}+\hspace{-0.8mm}\EE\|\sz^t -\sz^{t+1}\|^2\Big)\notag\\
&~~~~\le \max\{\frac{2}{\beta-\gamma}, (1+L^2)n\}(\EE L_{\beta}^{k-\uptau(\delta)-1} - \EE L_{\beta}^{k+1})
\end{align}
It implies that for any $k\ge 0$, it holds
\begin{align}
\EE \|g^{k}\|^2\le &C'(\EE L_{\beta}^{k} - \EE L_{\beta}^{k+\uptau(\delta)+2}), \label{eq:gk_forsum}
\end{align}
where $C':=\max\{\frac{2}{\beta-\gamma}, (1+L^2)n\}\frac{C}{(1-\delta)\pi_*}$. {\color{black} It can be simply verified that $C'=O(\frac{\uptau(\delta)+1}{(1-\delta)n\pi_*})$}
Let $\uptau':=\uptau(\delta)+2$.
Then for any $K>\uptau'$, summing \eqref{eq:gk_forsum} over $k\in\{K-\uptau',\cdots, {\rm mod}_{\uptau'}K\}$ gives
\begin{align}
\sum_{l=1}^{\lfloor\frac{K}{\uptau'}\rfloor} \EE\|g^{K-l\uptau'}\|^2&\le C'\left( \EE L_{\beta}^{ {\rm mod}_{\uptau'}K } - \EE L_{\beta}^{K}\right)\notag\\
&\le C'( L_{\beta}^0-\underline{f}), \label{eq:noncvx_rate_sum}
\end{align}
where the last inequality follows from the nondecreasing property of the sequence $(L_{\beta}^k)_{k\ge 0}$ and the fact that $(L_{\beta}^k)_{k\ge 0}$ is lower bounded by $\underline{f}$.

According to \eqref{eq:noncvx_rate_sum}, one has
\begin{align}
\min_{k\le K}\EE\|g^{k}\|^2 &\le \min_{1 \le l\le \lfloor\frac{K}{\uptau'}\rfloor}\EE\|g^{K-l\uptau'}\|^2\notag\\
 &\le  \frac{1}{ \lfloor\frac{K}{\uptau'}\rfloor}\sum_{l=1}^{\lfloor\frac{K}{\uptau'}\rfloor} \EE\|g^{K-l\uptau'}\|^2\notag\\
 &\le \frac{\uptau'C'}{K -\uptau'} ( L_{\beta}^0-\underline{f})\notag\\
 &\le \frac{C'(\uptau'+1)}{K}( L_{\beta}^0-\underline{f}),
\end{align}
where the constant  {\color{black} $C'(\uptau'+1)=O(\frac{\uptau(\delta)^2 + 1}{(1-\delta)n\pi_*})$}.
\end{proof}

We consider a reversible Markov chain on an undirected graph.
Recalling the definition of $\uptau(\delta)$ in \eqref{eq:def_J} and taking $\delta$ as $1/2$, one has $\uptau(\delta)\sim\frac{\ln n}{1-\lambda_2(\vP)}.$
That is, to guarantee that $\min_{k\le K}\EE\|g^{k}\|^2\le \epsilon$, 
\begin{align}
O\Big(\frac{1}{\epsilon}\cdot\Big(\frac{\ln^2 n}{(1-\lambda_2(\vP))^2} +1\Big)\Big)
\end{align}
iterations would be sufficient for Walkman using either  \eqref{eq:admm3y'} or \eqref{eq:admmy-linearized}.


 }

\ifCLASSOPTIONcaptionsoff
  \newpage
\fi

\bibliographystyle{IEEEbib}
\bibliography{reference}

\end{document}